\documentclass[english]{article}
\usepackage[T1]{fontenc}
\usepackage[latin1]{inputenc}
\usepackage{geometry}
\usepackage{float}
\usepackage{mathrsfs}
\usepackage{amsmath}
\usepackage{amssymb}
\usepackage{graphicx}
\usepackage{subfigure}

\makeatletter

\floatstyle{ruled}
\newfloat{algorithm}{tbp}{loa}
\providecommand{\algorithmname}{Algorithm}
\floatname{algorithm}{\protect\algorithmname}

\usepackage{amsthm}

\usepackage{mathrsfs}

\usepackage{amsfonts}

\usepackage{epsfig}

\usepackage{bm}

\usepackage{mathrsfs}

\usepackage{enumerate}

\@ifundefined{definecolor}{\@ifundefined{definecolor}
 {\@ifundefined{definecolor}
 {\usepackage{color}}{}
}{}
}{}

\usepackage{subfig}\usepackage[all]{xy}

\newtheorem{theorem}{Theorem}[section]
\newtheorem{lem}{Lemma}[section]
\newtheorem{rem}{Remark}[section]
\newtheorem{prop}{Proposition}[section]
\newtheorem{cor}{Corollary}[section]
\newcounter{hypA}

\newcounter{hypB}

\newcounter{hypD}
\newenvironment{hypD}{\refstepcounter{hypD}\begin{itemize}
 \item[({\bf D\arabic{hypD}})]}{\end{itemize}}

\usepackage{babel}\date{}

\usepackage{babel}

\makeatother

\usepackage{babel}

\begin{document}

\begin{center}

{\Large \textbf{Multilevel Particle Filters for the Non-Linear Filtering Problem in Continuous Time}}

\vspace{0.5cm}

BY AJAY JASRA$^{1}$, FANGYUAN YU$^{1}$ \& JEREMY HENG$^{2}$

{\footnotesize $^{1}$Computer, Electrical and Mathematical Sciences and Engineering Division, King Abdullah University of Science and Technology, Thuwal, 23955, KSA.}
{\footnotesize E-Mail:\,} \texttt{\emph{\footnotesize ajay.jasra@kaust.edu.sa, fangyuan.yu@kaust.edu.sa}}\\
{\footnotesize $^{2}$ESSEC Business School, Singapore, 139408, SG.}
{\footnotesize E-Mail:\,} \texttt{\emph{\footnotesize b00760223@essec.edu}}
\end{center}

\begin{abstract}
In the following article we consider the numerical approximation of the non-linear filter in continuous-time, where
the observations and signal follow diffusion processes.
Given access to high-frequency, but discrete-time observations, we resort to a first order time discretization of the non-linear filter, followed by
an Euler discretization of the signal dynamics. In order to approximate the associated discretized non-linear filter, one can use a particle filter. Under assumptions, this
can achieve a mean square error of $\mathcal{O}(\epsilon^2)$, for $\epsilon>0$ arbitrary, such that the associated cost is $\mathcal{O}(\epsilon^{-4})$.
We prove, under assumptions, that the multilevel particle filter of \cite{mlpf} can achieve a mean square error of $\mathcal{O}(\epsilon^2)$, for cost $\mathcal{O}(\epsilon^{-3})$.
This is supported by numerical simulations in several examples.
\end{abstract}

\section{Introduction}

The non-linear filtering problem in continuous-time is found in many applications in finance, economics and engineering; see e.g.~\cite{crisan_bain}.
We consider the case where one seeks to filter an unobserved diffusion process (the signal) with access to an observation trajectory that is, in theory, continuous in time and
following a diffusion process itself. The non-linear filter is the solution to the Kallianpur-Striebel formula (e.g.~\cite{crisan_bain}) and typically has no analytical solution.
This has led to a substantial literature on the numerical solution of the filtering problem; see for instance \cite{crisan_bain,delmoral}.

In practice, one has access to very high-frequency observations, but not an entire trajectory and this often means one has to time discretize the functionals associated
to the path of the observation and signal. This latter task can be achieved by using the approach in \cite{picard}, which is the one used in this article, but improvements exist; see for instance \cite{crisan_ortiz1,crisan_ortiz2}.
Even under such a time-discretization, such a filter is not available analytically, for most problems of practical interest. From here one must often discretize the dynamics of the signal (such as Euler), which in essence leads to a high-frequency discrete-time non-linear filter.
This latter object can be approximated using particle filters in discrete time, as in, for instance, \cite{crisan_bain}; this is the approach followed in this article.
Alternatives exist, such as unbiased methods \cite{fearn} and integration-by-parts, change of variables along with Feynman-Kac particle methods \cite{delmoral}, but, each of these schemes has its advantages and pitfalls versus the one followed in this paper. We refer to e.g.~\cite{crisan_ortiz2} for some discussion.

Particle filters generate $N$ samples (or particles) in parallel and sequentially approximate non-linear filters using sampling and resampling. The algorithms are very well understood mathematically; see for instance \cite{delmoral} and the references therein.
Given the particle filter approximation of the time-discretized filter, using an Euler method for the signal, one can expect that to obtain a mean squared error (MSE), relative to the true filter, of $\mathcal{O}(\epsilon^2)$, for $\epsilon>0$ arbitrary, the associated cost is $\mathcal{O}(\epsilon^{-4})$. This follows from standard results on discretizations and
particle filters. In a related context of regular, discrete time observations and dynamics, with the signal following a diffusion, \cite{mlpf} (see also \cite{mlpf_new}) show that when the MSE for a particle filter is $\mathcal{O}(\epsilon^2)$, the cost is $\mathcal{O}(\epsilon^{-3})$
and one can improve particle filters using the multilevel Monte Carlo (MLMC) method \cite{giles,giles1,hein}, as we now explain. 

MLMC is an approach which can help to approximate expectations w.r.t.~probability measures that are induced by discretizations, such as an Euler method. The idea is to create a telescoping sum representation of an expectation w.r.t.~an accurate discretization and interpolate with differences of expectations of increasingly coarse (in terms of the discretization) probability measures.
Then, if one can sample from appropriate couplings of the pairs of probability measures in the differences of the expectations, one can reduce the computational effort to achieve a given MSE. In the case of \cite{mlpf}, one can achieve a MSE $\mathcal{O}(\epsilon^2)$, for cost $\mathcal{O}(\epsilon^{-2.5})$ for a class of processes.

In this paper we apply the methodology of \cite{mlpf}, which combines particle filters with the MLMC methodology (termed the multilevel particle filter), to the non-linear filtering problem in continuous-time. The main issue is that in order to mathematically understand the application of this methodology to this new context, several new results are required. The main difference to the case of \cite{mlpf}, other than the processes involved, is the fact that
one averages over the data in the analysis of filters in continuous-time. This requires one to analyze the properties of several time-discretized Feynman-Kac semigroups, in order to verify the mathematical improvements of the approach (see also \cite{stolz}). Under assumptions, we prove that to achieve a MSE of $\mathcal{O}(\epsilon^2)$, one requires a cost of $\mathcal{O}(\epsilon^{-3})$.
This is verified in several numerical examples. We remark that the mathematical results are of interest beyond the context of this article, for instance, unbiased estimation; see \cite{glynn} for example.

This article is structured as follows. In Section \ref{sec:problem} we formalize the problem of interest.
Our approach is detailed in Section \ref{sec:mlpf}. The theoretical results are presented in Section \ref{sec:theory} and illustrated numerically in Section \ref{sec:numerics}. 
The proofs of our theoretical results are housed in the appendix.

\section{Problem}\label{sec:problem}
We introduce some notations in Section \ref{sec:notation} and continuous time non-linear filtering in Section \ref{sec:model}. 
Time-discretization of the non-linear filters is considered in Section \ref{sec:discrete_model}, followed by a discussion of multilevel estimation in this context
in Section \ref{sec:ml_estimation}.

\subsection{Notations}\label{sec:notation}

Let $(\mathsf{X},\mathcal{X})$ be a measurable space.
For $\varphi:\mathsf{X}\rightarrow\mathbb{R}$ we write $\mathcal{B}_b(\mathsf{X})$ as the collection of bounded measurable functions. 
Let $\varphi:\mathbb{R}^d\rightarrow\mathbb{R}$, $\textrm{Lip}_{\|\cdot\|_2}(\mathbb{R}^{d})$ denotes the collection of real-valued functions that are Lipschitz w.r.t.~$\|\cdot\|_2$ ($\|\cdot\|_p$ denotes the $\mathbb{L}_p-$norm of a vector $x\in\mathbb{R}^d$). That is, $\varphi\in\textrm{Lip}_{\|\cdot\|_2}(\mathbb{R}^{d})$ if there exists a $C<+\infty$ such that for any $(x,y)\in\mathbb{R}^{2d}$
$$
|\varphi(x)-\varphi(y)| \leq C\|x-y\|_2.
$$
We write $\|\varphi\|_{\textrm{Lip}}$ as the Lipschitz constant of a function $\varphi\in\textrm{Lip}_{\|\cdot\|_2}(\mathbb{R}^{d})$.
For $\varphi\in\mathcal{B}_b(\mathsf{X})$, we write the supremum norm $\|\varphi\|=\sup_{x\in\mathsf{X}}|\varphi(x)|$.
$\mathcal{P}(\mathsf{X})$  denotes the collection of probability measures on $(\mathsf{X},\mathcal{X})$.
For a measure $\mu$ on $(\mathsf{X},\mathcal{X})$
and a function $\varphi\in\mathcal{B}_b(\mathsf{X})$, the notation $\mu(\varphi)=\int_{\mathsf{X}}\varphi(x)\mu(dx)$ is used. 
$B(\mathbb{R}^d)$ denote the Borel sets on $\mathbb{R}^d$. $dx$ is used to denote the Lebesgue measure.
For $(\mathsf{X}\times\mathsf{Y},\mathcal{X}\vee\mathcal{Y})$ a measurable space and $\mu$ a non-negative measure on this space,
we use the tensor-product of functions notation for $(\varphi,\psi)\in\mathcal{B}_b(\mathsf{X})\times\mathcal{B}_b(\mathsf{X})$,
$\mu(\varphi\otimes\psi)=\int_{\mathsf{X}\times\mathsf{Y}}\varphi(x)\psi(y)\mu(d(x,y))$.
If $K:\mathsf{X}\times\mathcal{X}\rightarrow[0,\infty)$ is a non-negative operator and $\mu$ is a measure, we use the notations
$
\mu K(dy) = \int_{\mathsf{X}}\mu(dx) K(x,dy)
$
and for $\varphi\in\mathcal{B}_b(\mathsf{X})$, 
$
K(\varphi)(x) = \int_{\mathsf{X}} \varphi(y) K(x,dy).
$
For $A\in\mathcal{X}$, the indicator function is written as $\mathbb{I}_A(x)$.
$\mathcal{N}_s(\mu,\Sigma)$ (resp.~$\psi_s(x;\mu,\Sigma)$)
denotes an $s-$dimensional Gaussian distribution (density evaluated at $x\in\mathbb{R}^s$) of mean $\mu$ and covariance $\Sigma$. If $s=1$ we omit the subscript $s$. For a vector/matrix $X$, $X^*$ is used to denote the transpose of $X$.
For $A\in\mathcal{X}$, $\delta_A(du)$ denotes the Dirac measure of $A$, and if $A=\{x\}$ with $x\in \mathsf{X}$, we write $\delta_x(du)$. 
For a vector-valued function in $d-$dimensions (resp.~$d-$dimensional vector), $\varphi(x)$ (resp.~$x$) say, we write the $i^{\textrm{th}}-$component ($i\in\{1,\dots,d\}$) as $\varphi^{(i)}(x)$ (resp.~$x^{(i)}$). For a $d\times q$ matrix $x$, we write the $(i,j)^{\textrm{th}}-$entry as $x^{(ij)}$.
For $\mu\in\mathcal{P}(\mathsf{X})$ and $X$ a random variable on $\mathsf{X}$ with distribution associated to $\mu$, we use the notation $X\sim\mu(\cdot)$. For a finite set $A\in\mathcal{X}$, we write $\textrm{Card}(A)$ as the cardinality of $A$.

\subsection{Model}\label{sec:model}

Let $(\Omega,\mathcal{F})$ be a measurable space. On $(\Omega,\mathcal{F})$
consider the probability measure $\mathbb{P}$ and a pair of stochastic processes $\{Y_t\}_{t\geq 0}$, $\{X_t\}_{t\geq 0}$, with $Y_t\in\mathbb{R}^{d_y}$, $X_t\in\mathbb{R}^{d_x}$ for $(d_{y},d_x)\in\mathbb{N}^2$, and with $X_0=x_*\in\mathbb{R}^{d_x}$ given:
\begin{eqnarray}
dY_t & = & h(X_t)dt + dB_t \label{eq:obs}\\
dX_t & = & b(X_t)dt + \sigma(X_t)dW_t \label{eq:state}
\end{eqnarray}
where $h:\mathbb{R}^{d_x}\rightarrow\mathbb{R}^{d_y}$, $b:\mathbb{R}^{d_x}\rightarrow\mathbb{R}^{d_x}$, $\sigma:\mathbb{R}^{d_x}\rightarrow\mathbb{R}^{d_x\times d_x}$ with $\sigma$ non-constant and of full rank and $\{B_t\}_{t\geq 0}, \{W_t\}_{t\geq 0}$
are independent standard Brownian motions of dimension $d_y$ and $d_x$ respectively. 
The structure of the model is the `standard one' in non-linear filtering (e.g.~\cite{crisan_bain}) and we will restrict ourselves to this form. 
To minimize certain technical difficulties, the following assumption is made throughout the paper:
\begin{hypD}\label{hyp_diff:1}
We have:
\begin{enumerate}
\item{$\sigma^{(ij)}$ is bounded with $\sigma^{(ij)}\in\textrm{Lip}_{\|\cdot\|_2}(\mathbb{R}^{d_x})$, for all $(i,j)\in\{1,\dots,d_x\}^2$ and $a(x):=\sigma(x)\sigma(x)^*$ is uniformly elliptic for all $x\in\mathbb{R}^{d_x}$.}
\item{$(h^{(i)},b^{(j)})$ are bounded and $(h^{(i)},b^{(j)})\in\textrm{Lip}_{\|\cdot\|_2}(\mathbb{R}^{d_x})\times\textrm{Lip}_{\|\cdot\|_2}(\mathbb{R}^{d_x})$, for all $(i,j)\in\{1,\dots,d_y\}\times\{1,\dots,d_x\}$.}
\end{enumerate}
\end{hypD}


Let $\{\mathcal{F}_t\}_{t\geq 0}$ be a filtration on $\mathcal{F}$ such that $\{B_t\}_{t\geq 0}$ and $\{W_t\}_{t\geq 0}$
are independent standard Brownian motions. Let $T>0$ be an arbitrary positive real number 
and introduce the probability measure $\overline{\mathbb{P}}$, which is equivalent to $\mathbb{P}$ on $\mathcal{F}_T$, defined by the Radon-Nikodym derivative 
$$
Z_{T}:=\frac{d\mathbb{P}}{d\overline{\mathbb{P}}} = \exp\Big\{\int_{0}^T h(X_s)^*dY_s - \frac{1}{2}\int_{0}^Th(X_s)^*h(X_s)ds\Big\}.
$$
Under $\overline{\mathbb{P}}$, $\{X_t\}_{t\in[0,T]}$ follows the dynamics \eqref{eq:state} and independently $\{Y_t\}_{t\in [0,T]}$ is a standard Brownian motion. 
The solution to the Zakai equation is given by
$$
\gamma_{t}(\varphi) := \overline{\mathbb{E}}\Big[\varphi(X_t)\exp\Big\{\int_{0}^t h(X_s)^*dY_s - \frac{1}{2}\int_{0}^th(X_s)^*h(X_s)ds\Big\}\Big|\mathcal{Y}_t\Big]
$$
for $\varphi\in\mathcal{B}_b(\mathbb{R}^{d_x})$, where $\mathcal{Y}_t$ is the filtration generated by the process $\{Y_s\}_{0\leq s \leq t}$. Our objective is to, recursively in time, estimate the filter
$$
\eta_t(\varphi) := \frac{\gamma_{t}(\varphi)}{\gamma_{t}(1)}
$$
for $\varphi\in\mathcal{B}_b(\mathbb{R}^{d_x})$.

\subsection{Discretized Model}\label{sec:discrete_model}

For the rest of the article, we will fix a time horizon $T\in\mathbb{N}$.
In practice, we will have to work with a discretization of the model in \eqref{eq:obs}-\eqref{eq:state} for the following reasons.
\begin{enumerate}
\item{One only has access to a finite, but possibly very high frequency data.}
\item{$Z_{T}$ is typically unavailable analytically.}
\item{There may not be a non-negative and unbiased estimator of the transition densities induced by the model \eqref{eq:obs}-\eqref{eq:state}.}
\end{enumerate}
We will assume access to a path of data $\{Y_t\}_{0\leq t \leq T}$ which is observed at a high frequency, as mentioned above. 

Let $l\in\mathbb{N}_0$ denote a given level and consider an Euler discretization of step-size $\Delta_l=2^{-l}$:
\begin{eqnarray}
\widetilde{X}_{k\Delta_l} & = & \widetilde{X}_{(k-1)\Delta_l} + b(\widetilde{X}_{(k-1)\Delta_l})\Delta_l + \sigma(\widetilde{X}_{(k-1)\Delta_l})[W_{k\Delta_l}-W_{(k-1)\Delta_l}]\label{eq:disc_state}
\end{eqnarray}
for $k\in\{1,2,\dots,2^lT\}$, initialized at $\widetilde{X}_{0}=x_*$.
It should be noted that the Brownian motion in \eqref{eq:disc_state} is the same as in \eqref{eq:state} under both $\mathbb{P}$ and $\overline{\mathbb{P}}$.
Then, for $k\in\mathbb{N}_0$, we define the potential functions
$$
G_{k}^l(x_{k\Delta_l}) := \exp\Big\{h(x_{k\Delta_l})^*(y_{(k+1)\Delta_l}-y_{k\Delta_l})-\frac{\Delta_l}{2}h(x_{k\Delta_l})^*h(x_{k\Delta_l})\Big\}
$$
and note that their product over time
$$
Z_{T}^l(x_0,x_{\Delta_l},\dots,x_{T-\Delta_l}) := \prod_{k=0}^{2^lT-1}G_{k}^l(x_{k\Delta_l}) = \exp\Big\{\sum_{k=0}^{2^lT-1}\Big[h(x_{k\Delta_l})^*(y_{(k+1)\Delta_l}-y_{k\Delta_l})-\frac{\Delta_l}{2}h(x_{k\Delta_l})^*h(x_{k\Delta_l})\Big]\Big\}
$$
is simply a discretization of $Z_{T}$ (of the type of \cite{picard}).  
We can then define the time-discretized filter at time $t\in\mathbb{N}$ as
\begin{eqnarray}
\gamma_{t}^l(\varphi) & := & \overline{\mathbb{E}}\big[\varphi(X_t)Z_{t}^l(\widetilde{X}_0,\widetilde{X}_{\Delta_l},\dots,\widetilde{X}_{t-\Delta_l})|\mathcal{Y}_t\big] \notag\\
\eta_{t}^l(\varphi) & := & \frac{\gamma_{t}^l(\varphi)}{\gamma_{t}^l(1)}\label{eq:discretetime_filter}
\end{eqnarray}
where $\varphi\in\mathcal{B}_b(\mathbb{R}^{d_x})$, and for notational convenience we set $\eta_0^l(dx) := \delta_{x_{*}}(dx)$. 
To define the time-discretized filter between unit times, for $(p,t)\in \mathbb{N}_0\times\{\Delta_l,2\Delta_l,\dots,1-\Delta_l\}$, we set
\begin{eqnarray*}
\gamma_{p+t}^l(\varphi) & := & \overline{\mathbb{E}}\Big[\varphi(X_{p+t})Z_{p}^l(\widetilde{X}_0,\widetilde{X}_{\Delta_l},\dots,\widetilde{X}_{p-\Delta_l})\Big(\prod_{k=0}^{t\Delta_l^{-1}-1}G_{p\Delta_l^{-1}+k}^l(\widetilde{X}_{p+k\Delta_l^{-1}})\Big)\Big|\mathcal{Y}_{p+t}\Big] \\
\eta_{p+t}^l(\varphi) & := & \frac{\gamma_{p+t}^l(\varphi)}{\gamma_{p+t}^l(1)}
\end{eqnarray*}
where $\varphi\in\mathcal{B}_b(\mathbb{R}^{d_x})$ and $Z_{0}^l(x_{-\Delta_l})=1$. We note that the approach that we will consider is not constrained to the particular choice of $\Delta_l$; the one
chosen here is the one often used in previous works on MLMC and diffusions and we will continue with this convention. We also remark, as mentioned in the introduction,
that more precise discretization schemes exist and improvements in the approximation of $Z_T$ are possible, but not considered.

We now detail the sequential structure of the time-discretized filter \eqref{eq:discretetime_filter} that will facilitate its numerical approximation in Section \ref{sec:particle_filter}.
We write a path under discretization level $l$ on a unit time interval starting from $p\in\mathbb{N}_0$ as
$$
u_p^l := (x_{p},x_{p+\Delta_l},\dots,x_{p+1}) \in (\mathbb{R}^{d_x})^{\Delta_l^{-1}+1} =: E_l.
$$
For $\varphi\in\mathcal{B}_b(\mathbb{R}^{d_x})$, we define its extension on $E_l$, $\pmb{\varphi}^l:E_l\rightarrow\mathbb{R}$ by
$$
\pmb{\varphi}^l(x_0,x_{\Delta_l},\dots,x_{1}) := \varphi(x_1).
$$
We denote the product of potential functions on a time-discretized unit interval starting from $p\in\mathbb{N}_0$ as
$$
\mathbf{G}_p^l(u_p^l) := \prod_{k=0}^{\Delta_l^{-1}-1}G_{p\Delta_l^{-1}+k}^l(x_{p+k\Delta_l}).
$$
Let $M^l:\mathbb{R}^{d_x}\rightarrow\mathcal{P}(E_l)$ denote the joint Markov transition of a path $(x_0,x_{\Delta_l},\dots,x_{1})$ defined by the Euler discretization \eqref{eq:disc_state} with initialization $x\in\mathbb{R}^{d_x}$, i.e. for $\varphi\in\mathcal{B}_b(E_l)$, we have
$$
M^l(\varphi)(x) := \int_{E_l}\varphi(x_0,x_{\Delta_l},\dots,x_{1})\delta_x(dx_0)\Big[\prod_{k=1}^{\Delta_l^{-1}}\psi_{d_x}(x_{k\Delta_l};x_{(k-1)\Delta_l} + b(x_{(k-1)\Delta_l})\Delta_l ,a(x_{(k-1)\Delta_l})\Delta_l)\Big]d(x_{\Delta_l},\dots,x_{1}).
$$
For $p\in\mathbb{N}$, define the (prediction) operator $\Phi_p^l:\mathcal{P}(E_l)\rightarrow\mathcal{P}(E_l)$ as
\begin{equation}\label{eq:phi_rec}
\Phi_p^l(\mu)(\varphi) :=  \frac{\mu(\mathbf{G}_{p-1}^l\mathbf{M}^l(\varphi))}{\mu(\mathbf{G}_{p-1}^l)}
\end{equation}
for $(\mu,\varphi)\in\mathcal{P}(E_l)\times\mathcal{B}_b(E_l)$ (see e.g. \cite{delmoral2004} for more motivation and properties of this operator). 
To clarify notation in the preceding equation, we note that 
$$\mu(\mathbf{G}_{p-1}^l\mathbf{M}^l(\varphi))= \int_{E_l}\mu(d(x_{p-1},x_{p-1+\Delta_l},\dots,x_{p}))\mathbf{G}_{p-1}^l(x_{p-1},x_{p-1+\Delta_l},\dots,x_{p-\Delta_l})M^l(\varphi)(x_p).$$
The predictor can be defined as
$$
\pi_p^l(\varphi) := \int_{\mathbb{R}^{d_x}}\eta_{p}^l(dx) \int_{E_l} M^l(x,du)\varphi(u)
$$
for $(p,\varphi)\in\mathbb{N}_0\times\mathcal{B}_b(E_l)$, and can be shown to satisfy the recursion
$$
\pi_p^l(\varphi) = \Phi_p^l(\pi_{p-1}^l)(\varphi)
$$
for $p\in\mathbb{N}$. Lastly, note that one can also write the time-discretized filter at time $t\in\mathbb{N}$ in terms of the predictor at time $t-1$ using the Bayes' update
\begin{equation}\label{eq:pred_up}
\eta_t^l(\varphi) = \frac{\pi_{t-1}^l(\mathbf{G}_{t-1}^l\pmb{\varphi}^l)}{\pi_{t-1}^l(\mathbf{G}_{t-1}^l)}
\end{equation}
for $\varphi\in\mathcal{B}_b(\mathbb{R}^{d_x})$.

\subsection{Multilevel Estimation}\label{sec:ml_estimation}
In multilevel estimation, the basic idea is to approximate the following identity
\begin{equation}\label{eq:ml_identity}
\eta_{t}^{L}(\varphi) = \eta_{t}^{0}(\varphi) + \sum_{l=1}^L\Big([\eta_{t}^{l}-\eta_{t}^{l-1}](\varphi)\Big)
\end{equation}
for $\varphi\in\mathcal{B}_b(\mathbb{R}^{d_x})$, where $L\in\mathbb{N}$ denotes a given number of levels. 
The main objective is to construct a simulation method which can approximate the differences in the summands 
$[\eta_{t}^{l}-\eta_{t}^{l-1}](\varphi)$ in a dependent manner, so as to reduce the variance relative to approximating $\eta_{t}^{L}(\varphi)$ directly. The details are discussed in many articles, such as \cite{giles,giles1,cpf_clt,mlpf}.

The highest level $L$ is typically chosen to target a specific bias and this is the strategy considered here. In our context, there is a potential complication as $L$ also determines the level frequency of the data that are used - this is discussed below. 

In Section \ref{sec:particle_filter}, we will introduce particle filters to obtain Monte Carlo approximations of the filtering expectation $\eta_{t}^{0}(\varphi)$ at level $0$. In Section \ref{sec:cpf}, we then describe how to approximate the summands $[\eta_{t}^{l}-\eta_{t}^{l-1}](\varphi)$ 
for levels $l\in\{1,\ldots,L\}$ using a specific coupling of the particle filters at levels $l$ and $l-1$. Assuming that these coupled particle filters are simulated independently for each successive level, the resulting methodology is the multilevel particle filter \cite{mlpf}. 
The motivations and properties of this algorithm are studied in \cite{cpf_clt} and we refer the reader to that article for more details.

\section{Methodology}\label{sec:mlpf}

To simplify notation, throughout this section, we omit the $\widetilde{\cdot}$ notation from $X$ for the Euler discretization. 
This section concerns particle filtering approximations.
In Section \ref{sec:particle_filter}, we introduce particle filters and their approximation of time-discretized filtering expectations. 
Section \ref{sec:cpf} then details a coupling of particle filters between successive levels and its approximation of differences of time-discretized filtering expectations. 
Section \ref{sec:ml_est} ends this section by building on the previous two subsections to describe the multilevel particle filter and its multilevel estimator.

\subsection{Particle Filter}\label{sec:particle_filter}

We now consider approximating the time-discretized filter $\eta_{t}^l(\varphi)$ for some given level $l\in\mathbb{N}_0$ using a particle filter. Recall that the filtering expectation $\eta_t^0(\varphi)$ at level $l=0$ is needed to approximate the multilevel identity in \eqref{eq:ml_identity}. 
The objective of the particle filter (PF) is to provide an approximation of the formulae \eqref{eq:phi_rec} and \eqref{eq:pred_up}.
For a given $N\in\mathbb{N}$, the particle filter generates a system of random variables on $(E_l^N)^{n+1}$ at a time $n\in\mathbb{N}_0$ according to the probability measure
$$
\mathbb{Q}(d(u_0^{l,1:N},\dots,u_n^{l,1:N})) = \Big(\prod_{i=1}^N M^l(x_{*},du_0^{l,i})\Big)\prod_{i=1}^N\prod_{p=1}^n \Phi_p^l(\pi_{p-1}^{l,N})(du_p^{l,i})
$$
where for $\varphi\in\mathcal{B}_b(E_l)$
\begin{equation}\label{eq:empirical_measure_pred}
\pi_{p-1}^{l,N}(\varphi) := \frac{1}{N}\sum_{i=1}^N \varphi(u_{p-1}^{l,i})
\end{equation}
denotes the empirical measure at time $p-1$.
An algorithmic description of the particle filter is given in Algorithm \ref{alg:pf}. For $(t,\varphi)\in\mathbb{N}\times\mathcal{B}_b(\mathbb{R}^{d_x})$ one can approximate the time-discretized filter $\eta_{t}^l(\varphi)$ corresponding to step-size $\Delta_l$ via
\begin{equation}\label{eq:pf_est}
\eta_{t}^{l,N}(\varphi) := \frac{\pi_{t-1}^{l,N}(\mathbf{G}_{t-1}^l\pmb{\varphi}^l)}{\pi_{t-1}^{l,N}(\mathbf{G}_{t-1}^l)},
\end{equation}
which can be seen as a particle approximation of \eqref{eq:pred_up} using the empirical measure \eqref{eq:empirical_measure_pred} with $N$ particles.
For $(l,p,t,\varphi)\in\mathbb{N}\times\mathbb{N}_0\times\{\Delta_l,2\Delta_l,\dots,1-\Delta_l\}\times\mathcal{B}_b(\mathbb{R}^{d_x})$ one can also obtain a particle approximation of the time-discretized filter at time $p+t$, $\eta_{p+t}^l(\varphi)$, using
$$
\eta_{p+t}^{l,N}(\varphi) := \frac{\sum_{i=1}^N\Big(\prod_{k=0}^{t\Delta_l^{-1}-1}G_{p\Delta_l^{-1}+k}^l(x_{p+k\Delta_l}^{l,i})\Big)\varphi(x_{p+t}^{l,i})}{\sum_{i=1}^N\prod_{k=0}^{t\Delta_l^{-1}-1}G_{p\Delta_l^{-1}+k}^l(x_{p+k\Delta_l}^{l,i})}.
$$

\begin{algorithm}
\begin{enumerate}
\item{Initialize: For $i\in\{1,\dots,N\}$, generate $u_0^{l,i}$ from $M^l(x_{*},\cdot)$. Set $p=1$.}
\item{Update: For $i\in\{1,\dots,N\}$, generate $u_p^{l,i}$ from $\Phi_p^l(\pi_{p-1}^{l,N})(\cdot)$. Set $p=p+1$ and return to the start of 2.}
\end{enumerate}
\caption{Particle Filter.}
\label{alg:pf}
\end{algorithm}

\subsection{Coupled Particle Filter}\label{sec:cpf}
In this section, we describe how to approximate the summands $[\eta_{t}^{l}-\eta_{t}^{l-1}](\varphi)$ for levels $l\in\{1,\ldots,L\}$ 
in the multilevel identity \eqref{eq:ml_identity} using a specific coupling of particle filters at levels $l$ and $l-1$.
For $(l,\varphi)\in\mathbb{N}\times\mathcal{B}_b(\mathbb{R}^{d_x}\times\mathbb{R}^{d_x})$ (resp.~$(l,\varphi)\in\mathbb{N}\times\mathcal{B}_b(\mathbb{R}^{4d_x})$), we define $\pmb{\varphi}^l:E_l\times E_{l-1}\rightarrow\mathbb{R}$ (resp.~$\pmb{\varphi}^l:(E_l\times E_{l-1})^2\rightarrow\mathbb{R}$)
$$
\pmb{\varphi}^l\big((x_0,x_{\Delta_l},\dots,x_{1}),(x_0',x_{\Delta_{l-1}}',\dots,x_{1}')\big) := \varphi(x_1,x_1')
$$
(resp.~$\pmb{\varphi}^l\big((x_0,x_{\Delta_l},\dots,x_{1}),(x_0',x_{\Delta_{l-1}}',\dots,x_{1}'),(v_0,v_{\Delta_{l}},\dots,v_{1}),(v_0',v_{\Delta_{l-1}}',\dots,v_{1}')\big) := \varphi(x_1,x_1',v_1,v_1')$).
The following exposition closely follows \cite{cpf_clt}, with modifications to the context here.
Let $\check{P}^l:\mathbb{R}^{d_x}\times\mathbb{R}^{d_x}\rightarrow\mathcal{P}((\mathbb{R}^{d_x})^{\Delta_l^{-1}}\times \mathbb{R}^{d_x})^{\Delta_{l-1}^{-1}})$ be a Markov kernel, for paths $(x_{\Delta_l},\dots,x_1)$ and $(x_{\Delta_{l-1}}',\dots,x_{1}')$
constructed by using the same Brownian increments in the discretization \eqref{eq:disc_state} (see e.g.~\cite{giles} or \cite[Section 3.3]{mlpf_nc}).
Let $\check{M}^l:\mathbb{R}^{d_x}\times\mathbb{R}^{d_x}\rightarrow\mathcal{P}(E_l\times E_{l-1})$ be a Markov kernel defined for $(u,v,\varphi)\in\mathbb{R}^{d_x}\times\mathbb{R}^{d_x}\times\mathcal{B}_b(E_l\times E_{l-1})$ as
$$
\check{M}^l(\varphi)\Big((u,v)\Big) := \int_{E_l\times E_{l-1}} \varphi(u^l,u^{l-1})\delta_u(dx_0^l)\delta_v(dx_0^{l-1})\check{P}^l\Big((x_0^l,x_0^{l-1}),d((x_{\Delta_l}^l,\dots,x_1^l),(x_{\Delta_{l-1}}^{l-1},\dots,x_1^{l-1}))\Big).
$$
Note that for any $(u,v)\in\mathbb{R}^{d_x}\times\mathbb{R}^{d_x}$, $(A,B)\in B(E_l)\vee B(E_{l-1})$
$$
\check{M}^l\Big((u,v),A\times E_{l-1}\Big) = M^l(u,A)\quad\textrm{and}\quad
\check{M}^l\Big((u,v'),E_{l}\times B\Big) = M^{l-1}(v,B),
$$
i.e. $\check{M}^l$ is a coupling of $M^l$ and $M^{l-1}$ using common Brownian increments in the discretization scheme \eqref{eq:disc_state}.

We now describe a construction to couple the resampling steps for the particle filters at levels $l$ and $l-1$; this was considered in \cite{mlpf} and is based on a maximal coupling of the resampling indices to promote the sampling of common ancestors.
Let $(p,\varphi,\mu)\in\mathbb{N}\times \mathcal{B}_b(E_l\times E_{l-1})\times\mathcal{P}(E_{l}\times E_{l-1})$ and define the probability measure:
\begin{eqnarray}
\check{\Phi}_p^l(\mu)(\varphi) & := & \mu\Big(\{F_{p-1,\mu,l} \wedge F_{p-1,\mu,l-1}\} \mathbf{\check{M}}^l(\varphi)\Big)
+ \Big(1-
\mu\Big(\{F_{p-1,\mu,l} \wedge F_{p-1,\mu,l-1}\}\Big)
\Big) \times \nonumber \\ & & 
(\mu\otimes\mu)\Big(\Big\{\overline{F}_{p-1,\mu,l}\otimes \overline{F}_{p-1,\mu,l-1}\Big\}\mathbf{\bar{M}}^l(\varphi)\Big) \label{eq:phi_rec_cpf}
\end{eqnarray}
where for $(u,v)\in E_l\times E_{l-1}$
\begin{eqnarray*}
\overline{F}_{p-1,\mu,l}(u,v) &  = & \frac{F_{p-1,\mu,l}(u,v)-\{F_{p-1,\mu,l}(u,v)\wedge F_{p-1,\mu,l-1}(u,v)\}}{
\mu(F_{p-1,\mu,l}-\{F_{p-1,\mu,l}\wedge F_{p-1,\mu,l-1}\})} \\
\overline{F}_{p-1,\mu,l-1}(u,v) &  = & \frac{F_{p-1,\mu,l-1}(u,v)-\{F_{p-1,\mu,l}(u,v)\wedge F_{p-1,\mu,l-1}(u,v)\}}{
\mu(F_{p-1,\mu,l-1}-\{F_{p-1,\mu,l}\wedge F_{p-1,\mu,l-1}\})} 
\end{eqnarray*}
\begin{eqnarray*}
F_{p-1,\mu,l}(u,v) &  = & \check{G}_{p-1,\mu,l}(u)\otimes1 \\
F_{p-1,\mu,l-1}(u,v) &  = & 1\otimes \check{G}_{p-1,\mu,l-1}(v) \\
\check{G}_{p-1,\mu,l}(u) & = & \frac{\mathbf{G}_{p-1}^l(u)}{\mu(\mathbf{G}_{p-1}^l\otimes 1)} \\
\check{G}_{p-1,\mu,l-1}(v) & = & \frac{\mathbf{G}_{p-1}^{l-1}(v)}{\mu(1\otimes \mathbf{G}^{l-1}_{p-1})} 
\end{eqnarray*}
and for $((u,v),(u',v'))\in (E_l\times E_{l-1})\times (E_l\times E_{l-1})$ and $\varphi\in\mathcal{B}_b(E_l\times E_{l-1})$
$$
\bar{M}^l(\varphi)((u,v),(u',v')) = \check{M}^l(\varphi)(u,v').
$$
Recalling that the process \eqref{eq:state} starts at $x_{*}$, we define for $\varphi\in\mathcal{B}_b(E_l\times E_{l-1})$
$$
\check{\pi}_0^l(\varphi) := \check{M}^l(\varphi)\Big((x_*,x_*)\Big)
$$
and for $(p,\varphi)\in\mathbb{N}\times \mathcal{B}_b(E_l\times E_{l-1})$
\begin{equation}\label{eq:pred_rec_cpf}
\check{\pi}_p^l(\varphi) = \check{\Phi}_p^l(\check{\pi}_{p-1})(\varphi).
\end{equation}
Now it can be shown that (see \cite{mlpf}) that for $(t,\varphi)\in\mathbb{N}\times\mathcal{B}_b(\mathbb{R}^{d_x})$, we have 
\begin{equation}\label{eq:pred_up_cpf}
\eta_t^l(\varphi) = \frac{\check{\pi}_{t-1}^l((\mathbf{G}_{t-1}^l\pmb{\varphi}^l)\otimes 1)}{\check{\pi}_{t-1}^l(\mathbf{G}_{t-1}^l\otimes 1)}\quad\textrm{and}\quad
\eta_t^{l-1}(\varphi) = \frac{\check{\pi}_{t-1}^l(1\otimes (\mathbf{G}_{t-1}^{l-1}\pmb{\varphi}^{l-1}))}{\check{\pi}_{t-1}^l(1\otimes \mathbf{G}_{t-1}^{l-1})},
\end{equation}
i.e. the above coupling construction admits as marginal distributions the time-discretized filters at levels $l$ and $l-1$.
We note that this construction is by no means unique, nor, as discussed in \cite{cpf_clt} for the purposes of multilevel estimation optimal in any sense. When $d_x=1$, an alternative scheme which uses Wasserstein coupling is used in \cite{mlpf_new} (see also \cite{cpf_clt}). This latter
procedure is considered in Section \ref{sec:numerics}, but is not mathematically analyzed.


The objective of the coupled particle filter (CPF) is to provide an approximation of the formulae \eqref{eq:phi_rec_cpf} and \eqref{eq:pred_up_cpf}.
For $p\in\mathbb{N}_0$ set $w_p^l=(u_p^l,\bar{u}_p^{l-1})\in E_l\times E_{l-1}$.
For a given $N\in\mathbb{N}$, a CPF generates a system of random variables on $((E_l\times E_{l-1})^N)^{n+1}$ at a time $n\in\mathbb{N}_0$ according to the probability measure 
$$
\check{\mathbb{Q}}(d(w_0^{l,1:N},\dots,w_n^{l,1:N})) = \Big\{\prod_{i=1}^N \check{M}^l\Big((x_{*},x_{*}),dw_0^{l,i}\Big)\Big\}\prod_{i=1}^N\prod_{p=1}^n \check{\Phi}_p^l(\check{\pi}_{p-1}^{l,N})(dw_p^{l,i})
$$
where for $\varphi\in\mathcal{B}_b(E_l\times E_{l-1})$
\begin{equation}\label{eq:cpf_empirical_measure}
\check{\pi}_{p-1}^{l,N}(\varphi) := \frac{1}{N}\sum_{i=1}^N \varphi(w_{p-1}^{l,i})
\end{equation}
denotes the empirical measure at time $p-1$. To run a CPF, one must understand how to sample from $\check{\Phi}_p^l(\check{\pi}_{p-1}^{l,N})(\cdot)$ which is detailed in Algorithm \ref{alg:bayes_cpf}. An algorithmic description of the CPF is then described in Algorithm \ref{alg:cpf_max_coupling}.
Using a particle approximation of \eqref{eq:pred_up_cpf} with the empirical measure \eqref{eq:cpf_empirical_measure}, 
we can approximate $[\eta_{t}^{l}-\eta_{t}^{l-1}](\varphi)$, $\varphi\in\mathcal{B}_b(\mathbb{R}^{d_x})$ with
\begin{equation}\label{eq:cpf_est}
[\eta_{t}^{l}-\eta_{t}^{l-1}]^N(\varphi) := \frac{\check{\pi}_{t-1}^{l,N}((\mathbf{G}_{t-1}^l\pmb{\varphi}^l)\otimes 1)}{\check{\pi}_{t-1}^{l,N}(\mathbf{G}_{t-1}^l\otimes 1)} - \frac{\check{\pi}_{t-1}^{l,N}(1\otimes (\mathbf{G}_{t-1}^{l-1}\pmb{\varphi}^{l-1}))}{\check{\pi}_{t-1}^{l,N}(1\otimes \mathbf{G}_{t-1}^{l-1})}.
\end{equation}
For $(l,p,t,\varphi)\in\{2,3,\dots\}\times\mathbb{N}_0\times\{\Delta_{l-1},2\Delta_{l-1},\dots,1-\Delta_{l-1}\}\times\mathcal{B}_b(\mathbb{R}^{d_x})$ one can also estimate the differences of the time-discretized filter at time $p+t$, $[\eta_{p+t}^{l}-\eta_{p+t}^{l-1}](\varphi)$, as
$$
[\eta_{p+t}^{l}-\eta_{p+t}^{l-1}]^N(\varphi) := 
\frac{\sum_{i=1}^N\Big(\prod_{k=0}^{t\Delta_{l}^{-1}-1}G_{p\Delta_l^{-1}+k}^l(x_{p+k\Delta_l}^{l,i})\Big)\varphi(x_{p+t}^{l,i})}{\sum_{i=1}^N\prod_{k=0}^{t\Delta_{l}^{-1}-1}G_{p\Delta_l^{-1}+k}^l(x_{p+k\Delta_l}^{l,i})} - 
\frac{\sum_{i=1}^N\Big(\prod_{k=0}^{t\Delta_{l-1}^{-1}-1}G_{p\Delta_{l-1}^{-1}+k}^{l-1}(\bar{x}_{p+k\Delta_{l-1}}^{l-1,i})\Big)\varphi(\bar{x}_{p+t}^{l-1,i})}{\sum_{i=1}^N\prod_{k=0}^{t\Delta_{l-1}^{-1}-1}G_{p\Delta_{l-1}^{-1}+k}^{l-1}(\bar{x}_{p+k\Delta_{l-1}}^{l-1,i})}.
$$

\begin{algorithm}
\begin{enumerate}
\item{With probability $\check{\pi}_{p-1}^{l,N}\Big(\{F_{p-1,\check{\pi}_{p-1}^{l,N},l} \wedge F_{p-1,\check{\pi}_{p-1}^{l,N},l-1}\}\Big)$ generate $w_p\in (E_l\times E_{l-1})$ according to 
$$
\frac{\sum_{i=1}^NF_{p-1,\check{\pi}_{p-1}^{l,N},l}(w_{p-1}^{l,i})\wedge F_{p-1,\check{\pi}_{p-1}^{l,N},l-1}(w_{p-1}^{l-1,i})\check{M}^l\Big((x_p^{l,i},\bar{x}_p^{l-1,i}),\cdot\Big)}{\sum_{i=1}^NF_{p-1,\check{\pi}_{p-1}^{l,N},l}(w_{p-1}^{l,i})\wedge F_{p-1,\check{\pi}_{p-1}^{l,N},l-1}(w_{p-1}^{l-1,i})}.
$$
}
\item{Otherwise, generate $w_p\in (E_l\times E_{l-1})$ according to 
$$
\frac{1}{N^2}\sum_{i=1}^N\sum_{j=1}^N\Big\{\overline{F}_{p-1,\check{\pi}_{p-1}^{l,N},l}(w_{p-1}^{l,i})\otimes \overline{F}_{p-1,\check{\pi}_{p-1}^{l,N},l-1}(w_{p-1}^{l,j})\Big\}\check{M}^l\Big((x_p^{l,i},\bar{x}_p^{l-1,j}),\cdot\Big).
$$
}
\end{enumerate}
\caption{Sampling from $\check{\Phi}_p^l(\check{\pi}_{p-1}^{l,N})(\cdot)$.}
\label{alg:bayes_cpf}
\end{algorithm}

\begin{algorithm}
\begin{enumerate}
\item{Initialize: For $i\in\{1,\dots,N\}$, generate $w_0^{l,i}$ from $\check{M}^l\Big((x_{*},x_{*}),\cdot\Big)$. Set $p=1$.}
\item{Update: For $i\in\{1,\dots,N\}$, generate $w_p^{l,i}$ from $\check{\Phi}_p^l(\check{\pi}_{p-1}^{l,N})(\cdot)$ as described in Algorithm \ref{alg:bayes_cpf}. Set $p=p+1$ and return to the start of 2.}
\end{enumerate}
\caption{A Coupled Particle Filter.}
\label{alg:cpf_max_coupling}
\end{algorithm}

\subsection{Multilevel Particle Filter}\label{sec:ml_est}
We can now describe the multilevel particle filter (MLPF) using the developments in Sections \ref{sec:particle_filter} and \ref{sec:cpf}.

\begin{enumerate}
\item{Level 0: Run a PF as in Algorithm \ref{alg:pf} with $N_0$ samples, independently of all other levels.} 
\item{Level $l\in\{1,\dots,L\}$: Run a CPF (to approximate the time-discretized filters at levels $l$ and $l-1$) as in Algorithm \ref{alg:cpf_max_coupling} with $N_l$ samples, independently of all other levels.}
\end{enumerate}
Using the multilevel identity \eqref{eq:ml_identity}, an estimator of expectations $\eta_{t}^{L}(\varphi)$ for $t\in\mathbb{N}, \varphi\in\mathcal{B}_b(\mathbb{R}^{d_x})$ with respect to the time-discretized filter at the highest level $L$ is then given by
\begin{equation}\label{eq:ml_est}
\eta_{t}^{L,ML}(\varphi) := \eta_{t}^{0,N_0}(\varphi) + \sum_{l=1}^L [\eta_{t}^{l}-\eta_{t}^{l-1}]^{N_l}(\varphi)
\end{equation}
where $\eta_{t}^{0,N_0}(\varphi)$ is the estimator in \eqref{eq:pf_est} at level $l=0$ with $N=N_0$ samples, 
and $[\eta_{t}^{l}-\eta_{t}^{l-1}]^{N_l}(\varphi)$ is the estimator in \eqref{eq:cpf_est} with $N=N_l$ samples.
We now consider estimating expectations with respect to the time-discretized filter at level $L$ between unit times.
For $(l,p,t,\varphi)\in\{1,\dots,L-1\}\times\mathbb{N}_0\times\{\Delta_{l},2\Delta_{l},\dots,1-\Delta_{l}\}\times\mathcal{B}_b(\mathbb{R}^{d_x})$, one can approximate $\eta_{p+t}^{L}(\varphi)$ as a by-product of the above procedure using
\begin{eqnarray}\label{eq:ml_est_int}
\eta_{p+t}^{L,l,ML}(\varphi) & := & \sum_{m=l+1}^L\Bigg\{
\frac{\sum_{i=1}^{N_m}\Big(\prod_{k=0}^{t\Delta_{m}^{-1}-1}G_{p\Delta_m^{-1}+k}^m(x_{p+k\Delta_m}^{l,i})\Big)\varphi(x_{p+t}^{m,i})}{\sum_{i=1}^{N_m}\prod_{k=0}^{t\Delta_{m}^{-1}-1}G_{p\Delta_m^{-1}+k}^l(x_{p+k\Delta_m}^{l,i})}- \nonumber \\ & &
\frac{\sum_{i=1}^{N_m}\Big(\prod_{k=0}^{t\Delta_{m-1}^{-1}-1}G_{p\Delta_{m-1}^{-1}+k}^{m-1}(\bar{x}_{p+k\Delta_{m-1}}^{m-1,i})\Big)\varphi(\bar{x}_{p+t}^{m-1,i})}{\sum_{i=1}^{N_m}\prod_{k=0}^{t\Delta_{m-1}^{-1}-1}G_{p\Delta_{m-1}^{-1}+k}^{m-1}(\bar{x}_{p+k\Delta_{m-1}}^{m-1,i})}\Bigg\}+ \nonumber\\ & & 
\frac{\sum_{i=1}^{N_l}\Big(\prod_{k=0}^{t\Delta_l^{-1}-1}G_{p\Delta_l^{-1}+k}^l(x_{p+k\Delta_l}^{l,i})\Big)\varphi(x_{p+t}^{l,i})}{\sum_{i=1}^{N_l}\prod_{k=0}^{t\Delta_l^{-1}-1}G_{p\Delta_l^{-1}+k}^l(x_{p+k\Delta_l}^{l,i})} \label{eq:ml_est_int}
\end{eqnarray}
where the $m^{\textrm{th}}-$term of the summand on the R.H.S.~has been obtained by a CPF at level $m$, and the second term in the sum on the R.H.S.~is the level $l$ particles from the CPF run targeting $(\eta_t^l,\eta_t^{l-1})$.

\section{Theoretical Results}\label{sec:theory}

We consider the estimator \eqref{eq:ml_est} in our analysis. The estimator \eqref{eq:ml_est_int} can also be analyzed with the same approach with only additional notational complications.
$\overline{\mathbb{E}}$ is used to denote expectations w.r.t.~the simulated process, which averages over the dynamics of the data, under the probability measure $\overline{\mathbb{P}}$.
The proofs needed for the following result are given in the appendix. Below, we write $A:=\{(l,q)\in\{1,\dots,L\}^2:l\neq q\}$. The following theorem gives a bound on the MSE.

\begin{theorem}\label{theo:main}
Assume (D\ref{hyp_diff:1}). Then for any $t\in\mathbb{N}_0$, there exists a $C<+\infty$  such that for any $L\in\{1,2,\dots\}$, $((N_0,\dots,N_L),\varphi)\in\mathbb{N}^{L+1}\times\mathcal{B}_b(\mathbb{R}^{d_x})\cap\textrm{\emph{Lip}}_{\|\cdot\|_2}(\mathbb{R}^{d_x})$
$$
\overline{\mathbb{E}}[(\eta_{t}^{L,ML}(\varphi)-\eta_t(\varphi))^2] \leq 
$$
$$
C(\|\varphi\|+\|\varphi\|_{\textrm{\emph{Lip}}})^2\left(\sum_{l=0}^L\left(\frac{\Delta_l^{1/2}}{N_l} 
+ \frac{\Delta_l^{1/4}}{N_l^{3/2}} 
\right)
+ \sum_{l=1}^L\sum_{q=1}^L\mathbb{I}_{A}(l,q)
\left(\frac{\Delta_l^{1/4}}{N_l^{1/2}} 
+ \frac{\Delta_l^{1/8}}{N_l^{3/4}}
\right)
\left(\frac{\Delta_q^{1/4}}{N_q^{1/2}} 
+ \frac{\Delta_q^{1/8}}{N_q^{3/4}}
\right)
+ \Delta_L\right).
$$
\end{theorem}

\begin{proof}
We consider
\begin{align}\label{eq:thm_decomp}
\overline{\mathbb{E}}[(\eta_{t}^{L,ML}(\varphi)-\eta_t(\varphi))^2] \leq 2\overline{\mathbb{E}}[[\eta_{t}^{L,ML}-\eta_t^L](\varphi)^2] + 2\overline{\mathbb{E}}[[\eta_t^L-\eta_t](\varphi)^2].
\end{align}
For the right-most term on the R.H.S. of \eqref{eq:thm_decomp}, one can use Remark \ref{rem:theo_bias} (in the appendix).
For the left-most term on the R.H.S. of \eqref{eq:thm_decomp}, one has by the $C_2-$inequality that 
\begin{align}\label{eq:thm_decomp_further}
\overline{\mathbb{E}}[[\eta_{t}^{L,ML}-\eta_t^L](\varphi)^2] \leq 2\overline{\mathbb{E}}[[\eta_{t}^{0,N_0}-\eta_{t}^{0}](\varphi)^2] + 
2\overline{\mathbb{E}}\left[\left(\sum_{l=1}^L\{[\eta_{t}^{l}-\eta_{t}^{l-1}]^{N_l}-[\eta_{t}^{l}-\eta_{t}^{l-1}]\}(\varphi)\right)^2\right].
\end{align}
The first term on the R.H.S. of \eqref{eq:thm_decomp_further} can be dealt with using Remark \ref{rem:idiot_referee} (in the appendix). 
For the second term on the R.H.S. of \eqref{eq:thm_decomp_further}, one can multiply out the brackets
and apply Proposition \ref{lem:l2_cpf_l2} for the terms with second moments. 
For the cross product terms, which \emph{are only independent conditional on the data}, one can use the Cauchy-Schwarz inequality and Proposition \ref{lem:l2_cpf_l2}.
\end{proof}

\begin{rem}
It is remarked that Theorem \ref{theo:main} is given in terms of the probability measure $\overline{\mathbb{P}}$ not $\mathbb{P}$; in the literature, one 
would typically state the result under $\mathbb{P}$. However, as 
$$
\mathbb{E}((\eta_{t}^{L,ML}(\varphi)-\eta_t(\varphi))^2] = \overline{\mathbb{E}}[Z_t(\eta_{t}^{L,ML}(\varphi)-\eta_t(\varphi))^2],
$$
by following the proofs, for instance of Lemmata \ref{lem:l2_cpf} and \ref{lem:l2_cpf_contr}, one can still deduce the same result under $\mathbb{P}$.
The calculations are of a fairly standard nature and are omitted for brevity.
\end{rem}

\begin{rem}
We note that the constant $C$ in Theorem \ref{theo:main} depends upon $t$. As seen in \cite{cpf_clt}, the task of bounding the asymptotic variance uniformly in $t$ (for models as in \cite{mlpf})
 is particularly difficult and one expects even more arduous calculations for the finite-sample variance. All of our below discussion does not consider $t$, although this is of course a very important issue. 
\end{rem}

We note that if one considers \eqref{eq:pf_est}, then the MSE associated to this estimator can be upper-bounded (using the $C_2-$inequality, Remark \ref{rem:ext_lp_pf} along with Lemma \ref{lem:res1} and Remark \ref{rem:theo_bias}) by
\begin{equation}\label{eq:pf_mse}
C(\|\varphi\|+\|\varphi\|_{\textrm{Lip}})^2\left(\frac{1}{N} + \Delta_L\right).
\end{equation}
Note that the bias term is $\mathcal{O}(\Delta_L)$ and not $\mathcal{O}(\Delta_L^2)$ (as in e.g.~\cite{mlpf}) as our results averages over the uncertainty in the data, as is often done in the literature in the analysis of continuous-time particle filters (e.g.~\cite{crisan_bain}). The order of the bias can be improved by using higher-order discretization methods.

Let $\epsilon>0$ be arbitrary. From Theorem \ref{theo:main}, to obtain a bound on the MSE of $\mathcal{O}(\epsilon^2)$, one can first choose the highest level $L$ so that the bias is $\Delta_L=\mathcal{O}(\epsilon^2)$. Then by having $N_l=\mathcal{O}(\epsilon^{-2}\Delta_L^{-1/4}\Delta_{l}^{3/4})$ samples at level $l$, the upper-bound in Theorem \ref{theo:main} would be $\mathcal{O}(\epsilon^2)$. The associated cost to achieve this MSE is $\mathcal{O}(\sum_{l=0}^L \Delta_l^{-1}N_l)=\mathcal{O}(\epsilon^{-3})$. 
For the estimator \eqref{eq:pf_est}, it follows from the upper-bound in \eqref{eq:pf_mse} that choosing the level $L$ to keep the bias $\Delta_L=\mathcal{O}(\epsilon^2)$ and having $N=\mathcal{O}(\epsilon^{-2})$ number of samples would require a cost of $\mathcal{O}(\Delta_L^{-1}N) = \mathcal{O}(\epsilon^{-4})$. We note that if the diffusion coefficient $\sigma$ in \eqref{eq:state} were constant, 
one would expect that the MSE associated to \eqref{eq:ml_est} to be upper-bounded by a term of $\mathcal{O}(\sum_{l=0}^L\frac{\Delta_l}{N_l} + \Delta_L)$. By choosing $N_l=\mathcal{O}(\epsilon^{-2}\Delta_{l}L)$, one can then show that the cost to achieve a MSE of $\mathcal{O}(\epsilon^2)$ is $\mathcal{O}(\epsilon^{-2}\log(\epsilon)^2)$ for the MLPF method.
For the PF, again the cost is $\mathcal{O}(\Delta_L^{-1}N) = \mathcal{O}(\epsilon^{-4})$.

An issue that may appear in practice is that by increasing the highest level $L$, one must have access to data that was observed at a frequency of $2^{-L}$. This could create a bottleneck for the multilevel procedure as one cannot exceed the frequency at which the data was observed. A possible remedy is to linearly interpolate the data, in which case, one may want to consider the robust filter (see \cite{clark} and \cite[Chapter 5]{crisan_bain}).

\section{Numerical Results}\label{sec:numerics}
In the following, we introduce four models in Section \ref{sec:numerical_model}, detail some implementation settings in Section \ref{sec:numerical_settings}
and conclude with our numerical findings in Section \ref{sec:numerical_results}.

\subsection{Models}\label{sec:numerical_model}
We set the model dimensions as $d_x=d_y=1$ and the initialization as $X_0=x_{*}=0$.
We will consider four different models for the signal and the observation function $h(x)=x$ in all cases. 
Data were generated from the process under the probability measure $\overline{\mathbb{P}}$.

\textbf{Ornstein-Uhlenbeck}. This process is defined as the solution of 
\begin{equation*}
dX_{t} = \theta(\mu-X_{t})dt + \sigma dW_{t}.
\end{equation*}
The parameters in the example are taken as $\theta=1$, $\mu=0$ and $\sigma=0.5$. 

\textbf{Langevin dynamics}. The (overdamped) Langevin dynamics is defined by the stochastic differential equation (SDE)
\begin{equation*}
dX_{t} = \frac{1}{2}\nabla \log\pi(X_{t})dt + dW_{t} 
\end{equation*}
where $\pi(x)$ denotes the stationary probability density function. We select the latter as the $t-$distribution with 
$10$ degrees of freedom (with zero location, unit scale). 

\textbf{Geometric Brownian motion}. This process is defined by the following SDE
\begin{equation*}
dX_{t} = \mu X_{t}dt + \sigma X_{t} dW_{t}
\end{equation*}
We set the model parameters as $\sigma = 0.2$ and $\mu = 0.02$. 

\textbf{SDE with a non-linear diffusion term}. Lastly, we consider the SDE
\begin{equation*}
dX_{t} = \theta(\mu-X_{t})dt + \frac{1}{\sqrt{1+X_{t}^{2}}} dW_{t} 
\end{equation*}
with parameters $\theta=0$ and $\mu=0$. 

We remark that none of these models satisfy (D\ref{hyp_diff:1}). This assumption, which can hold for certain models,
is mainly used to keep the mathematical proofs at a sensible length and simultaneously providing a formal proof of the properties that 
we expect to see in more generality for problems of practical interest. This will be illustrated below.

\subsection{Simulation Settings}\label{sec:numerical_settings}

We will compare the standard approach based on the PF (Section \ref{sec:particle_filter}) to the MLPF (Section \ref{sec:ml_est}). 
We will also consider another multilevel method proposed in \cite{mlpf_new}.
For the case of a constant diffusion coefficient (resp.~non-constant) of the signal, we expect that one can set $\Delta_L=\mathcal{O}(\epsilon^2)$ and 
$N_l=\mathcal{O}(\epsilon^{-2}\Delta_{l}^{3/2})$ (resp.~$N_l=\mathcal{O}(\epsilon^{-2}\Delta_{l}L)$) in the approach of \cite{mlpf_new} 
to achieve a MSE of $\mathcal{O}(\epsilon^2)$ with a cost of $\mathcal{O}(\epsilon^{-2})$ (resp.~$\mathcal{O}(\epsilon^{-2}\log(\epsilon)^2)$). 
We note that, to our knowledge, there is no proof of this result in our context and it is a topic to be considered in future work. 

In the following, we consider computing the expected value of the signal, i.e. $\varphi(x)=x$, at $T=100$ time units. 
For all examples, the multilevel estimator \eqref{eq:ml_est} is considered at levels $L\in\{4,\dots,9\}$. 
For the Ornstein-Uhlenbeck model, the ground truth is computed using a Kalman filter. 
For all other examples, the results from a particle filter at level $L=10$ with $N=100\times 2^{10}$ particles is used as an approximation of the ground truth. 
For each level $l$ of the PF, $N_{l}=100\times\Delta_l$ particles are used. The number of particles in the MLPF and the multilevel method of \cite{mlpf_new} 
are specified using the discussion in Section \ref{sec:theory} and the above paragraph, respectively.
All results are averaged over multiple runs. 
We employ adaptive resampling for all three approaches; 
resampling is performed whenever the effective sample size of the coarse filter (when using a CPF) falls below half the number of samples simulated.

\subsection{Results}\label{sec:numerical_results}
Our numerical results are summarized in Figure \ref{fig:res} in terms of log-log (base 10) plots of the cost against MSE. 
These plots agree with our theoretical findings in Section \ref{sec:theory}.
For all models, the slope of the plots in Figure \ref{fig:res} corresponding to PF is around $-2$, which matches with a cost of 
$\mathcal{O}(\epsilon^{-4})$ to achieve a MSE of $\mathcal{O}(\epsilon^{2})$.
For models (with constant diffusion coefficient) on the first row Figure \ref{fig:res}, we observe slopes close to $-1$, which corresponds roughly to a cost of 
$\mathcal{O}(\epsilon^{-2}\log(\epsilon)^2)$ to obtain a MSE of $\mathcal{O}(\epsilon^{2})$. 
Similarly, for models (with non-constant diffusion coefficient) on the second row of Figure \ref{fig:res}, slopes of approximately $-1.5$ agrees 
with a cost of $\mathcal{O}(\epsilon^{-3})$ for a MSE of $\mathcal{O}(\epsilon^{2})$.
Lastly, we also observe that the conjectured improvements of the method in \cite{mlpf_new} seem to be confirmed in these examples. 


\begin{figure}[!htbp]\centering
\subfigure[Ornstein-Uhlenbeck]{\includegraphics[width=6cm,height=5.5cm]{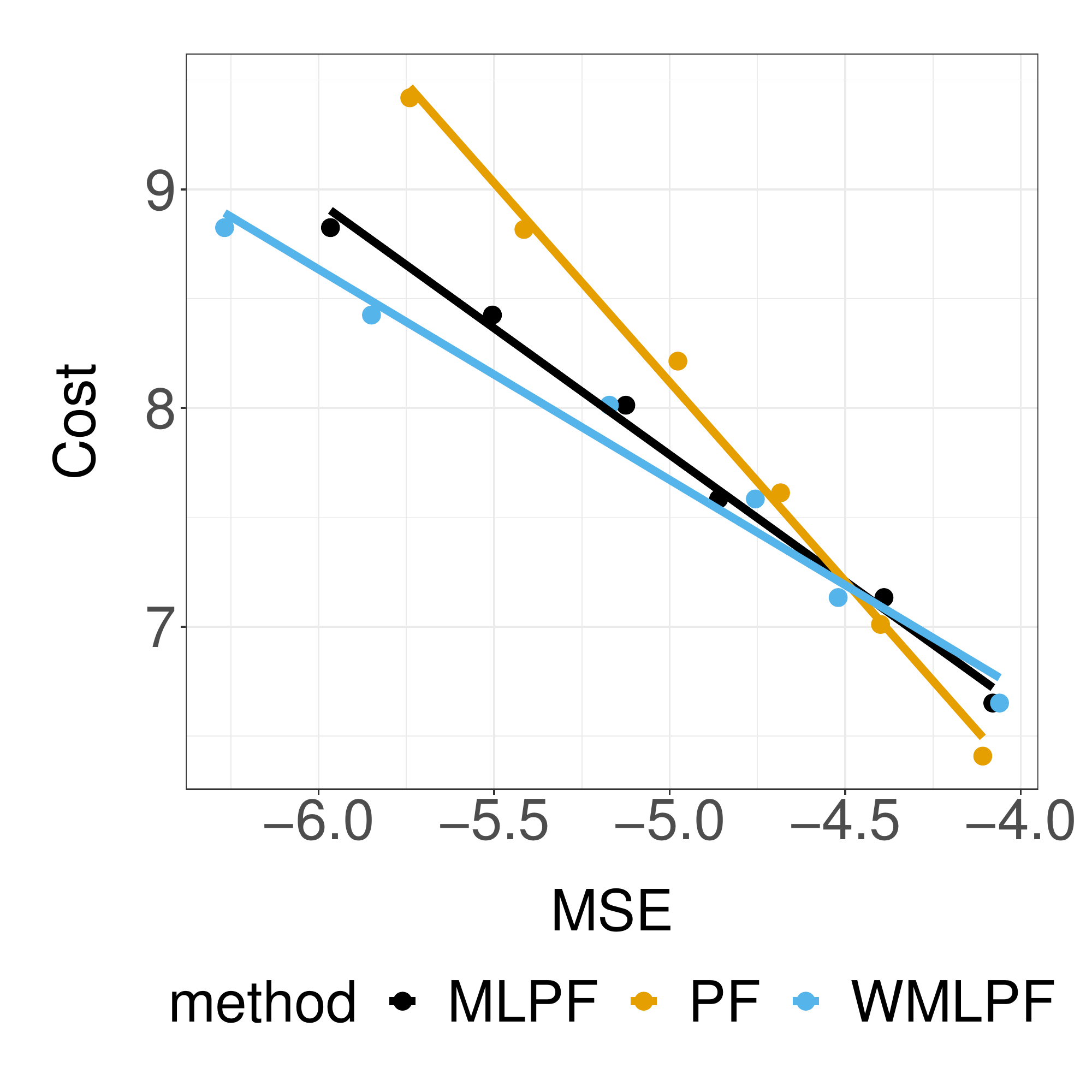}}
\subfigure[Langevin dynamics]{\includegraphics[width=6cm,height=5.5cm]{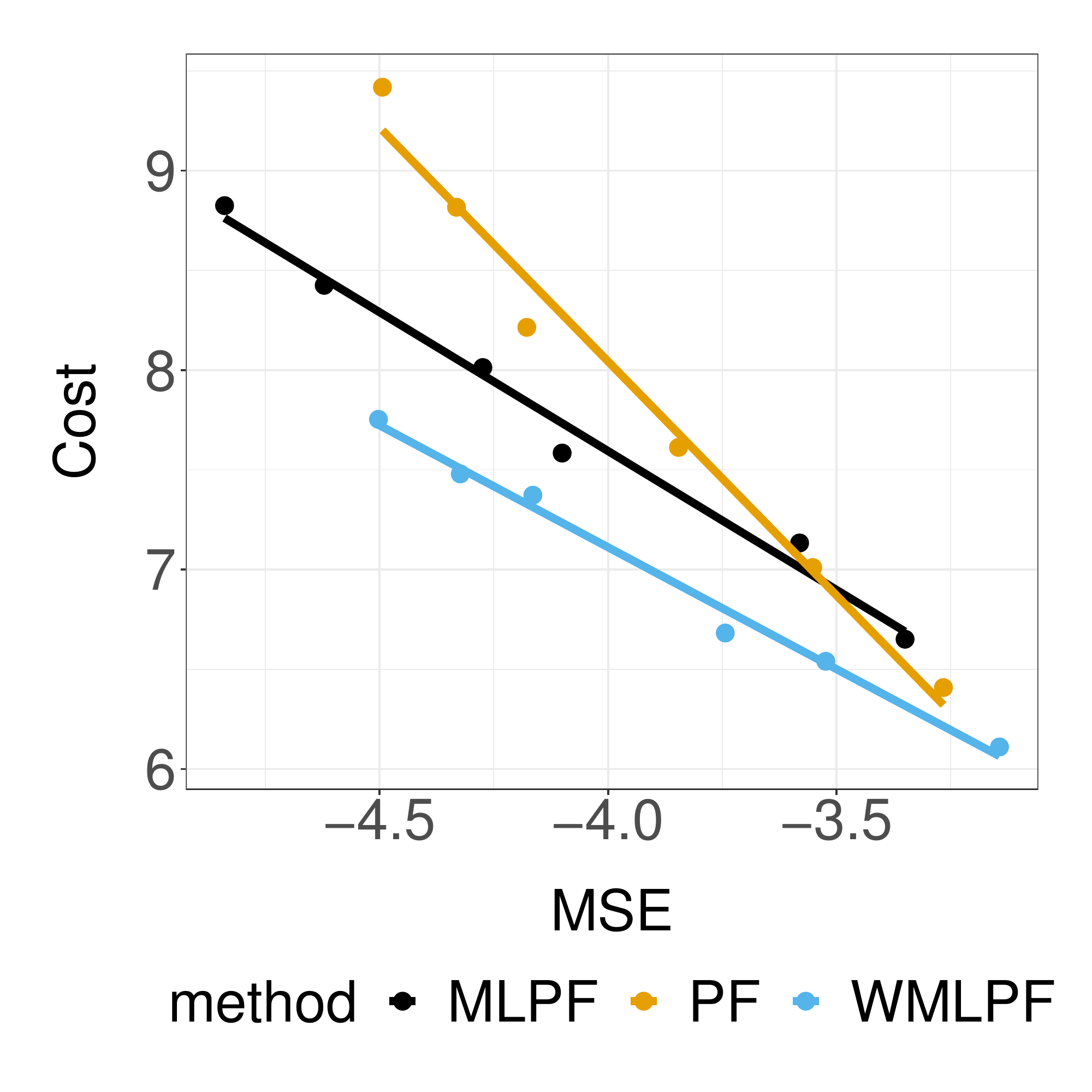}}
\subfigure[SDE with a non-linear diffusion term]{\includegraphics[width=6cm,height=5.5cm]{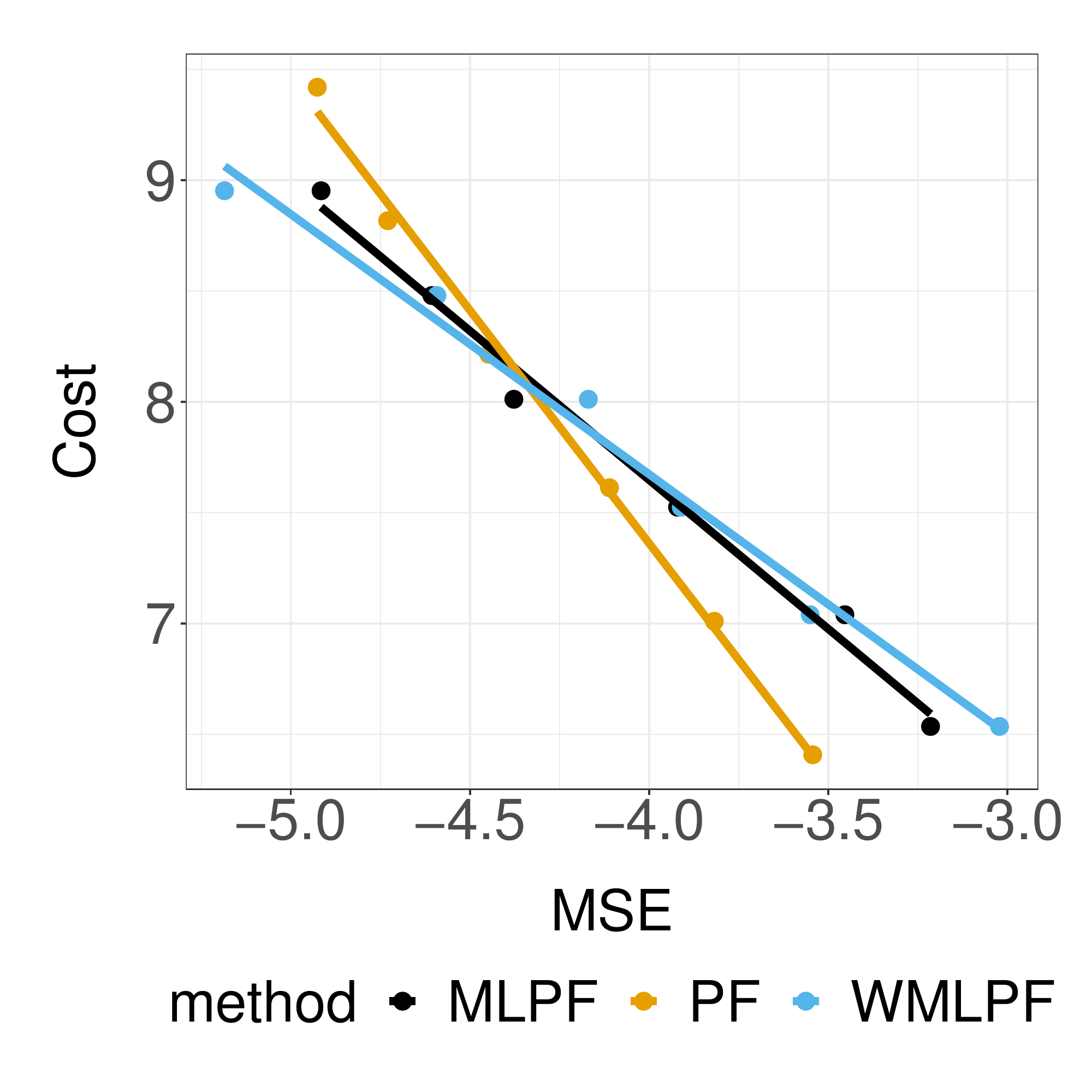}}
\subfigure[Geometric Brownian motion]{\includegraphics[width=6cm,height=5.5cm]{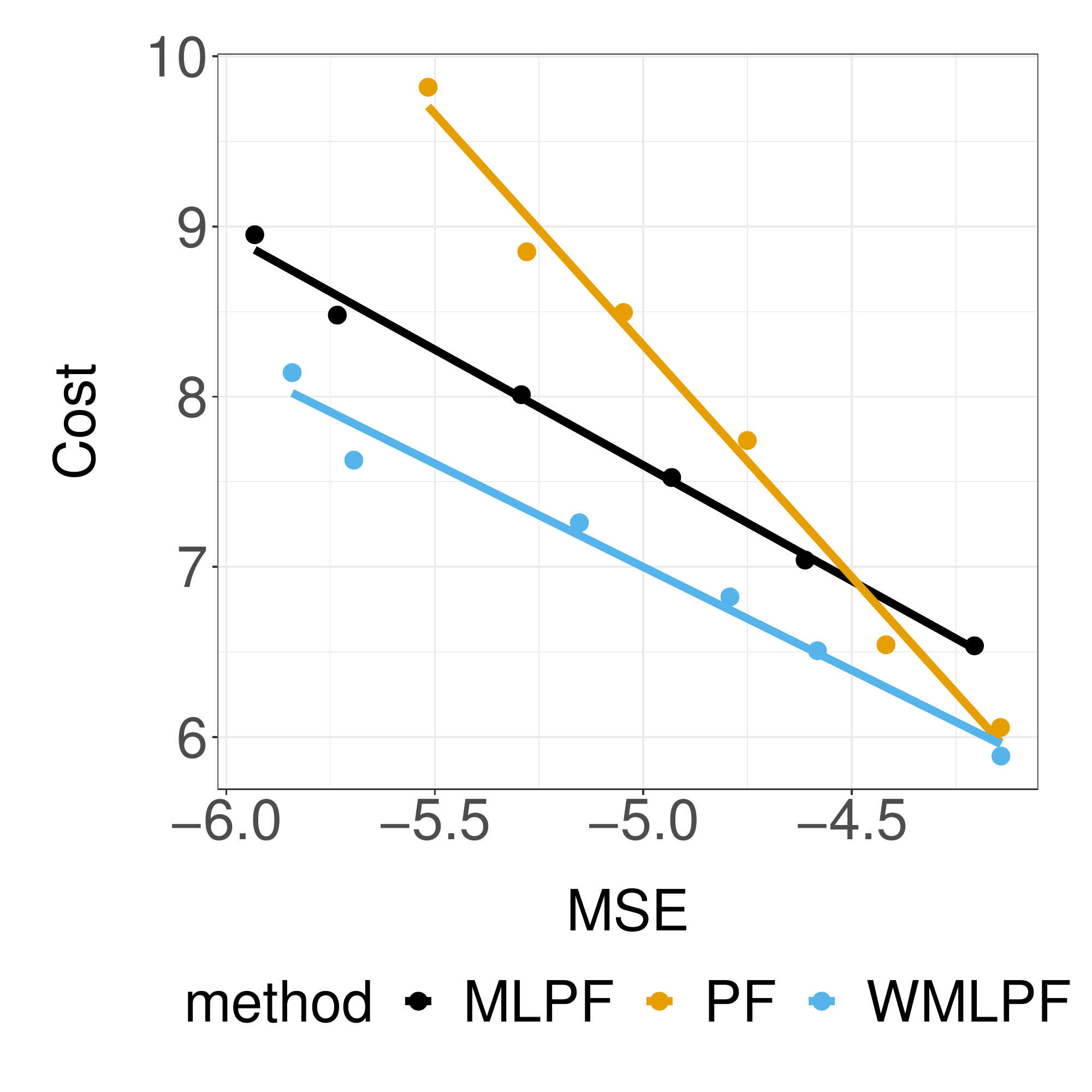}}
\caption{Log-log plots of cost against MSE for the PF (mustard), MLPF (black) and the multilevel method of \cite{mlpf_new} (baby blue).\label{fig:res}}
\end{figure}

\subsubsection*{Acknowledgements}

AJ was supported by KAUST baseline funding. We thank two referees for comments that have greatly improved the article.


\appendix

\section{Proofs}

\subsection{Objective and Structure}

The main objective of this appendix is to provide a meaningful bound, in terms of $l$ and $N$, on
$$
\overline{\mathbb{E}}[\{[\eta_{t}^{l}-\eta_{t}^{l-1}]^{N}-[\eta_{t}^{l}-\eta_{t}^{l-1}]\}(\varphi)^2].
$$
This is the main result which we will need to prove Theorem \ref{theo:main}. The typical way that this can be achieved
is to consider the predictor $\check{\pi}_t^{l,N}$ in \eqref{eq:cpf_est}. Our strategy to consider this latter object 
first in Lemma \ref{lem:l2_cpf}, which provides an upper-bound in terms of $N$ and the expectation of some operators.
The remainder of the proof is then concerned with the control of these expectations in terms of $l$; this result is given in Lemma \ref{lem:l2_cpf_contr}.
Finally these results are used to prove the bound of interest in Proposition \ref{lem:l2_cpf_l2}. The remaining technical developments are all used
to achieve these results. At a first reading, one can then proceed as just looking at Lemmata \ref{lem:l2_cpf} and \ref{lem:l2_cpf_contr} followed by 
Proposition \ref{lem:l2_cpf_l2}.

The structure of this appendix is as follows. In Section \ref{sec_app:note} we provide some additional notations and results which will be used in the appendix.
In Section \ref{app:cpf}, we give results on the coupled particle filter and in particular Lemma \ref{lem:l2_cpf}. Within this section are two additional subsections;
Sections \ref{app:cpf1} and \ref{app:cpf22}. In Section \ref{app:cpf1}, we consider the `coupling ability' of the coupled particle filter. That is, how close particle pairs
are, in terms of $\Delta_l$, which is the critical property of CPFs and indeed multilevel methods. Our results in this section are simply specializations of the results already proved
in \cite{mlpf}. In Section \ref{app:cpf22}, we give the important results Lemma \ref{lem:l2_cpf_contr} and Proposition \ref{lem:l2_cpf_l2}.
The results in the afore-mentioned appendices depend themselves on some important properties of non-linear filtering problems in continuous time and particle filters.
The non-linear filtering properties are given in Section \ref{app:nl_filt} and this can be read more-or-less linearly. The results are fairly well understood in the literature, but are
provided for completeness of the article. The particle filter is considered in Section \ref{app:pf} and the results are essentially the standard ones, for instance in \cite{delmoral}, given
in the context of this paper. One way to read the appendix is, begin with Section \ref{sec_app:note} and then to read, linearly, Section \ref{app:cpf} in its entirety, accepting the results in Sections \ref{app:nl_filt}  and \ref{app:pf}
and those latter sections can then be read (which can both be read linearly). Another way is Section \ref{sec_app:note}-\ref{app:nl_filt}-\ref{app:pf}-\ref{app:cpf}.

\subsection{Some Notations}\label{sec_app:note}

Some operators are now defined. For $(l,p,n)\in\mathbb{N}_0^3$, $n>p$, $(u_p,\varphi)\in E_l\times\mathcal{B}_b(E_l)$, let 
$$
\mathbf{Q}_{p,n}^l(\varphi)(u_p) := \int \varphi(u_n)\Big(\prod_{q=p}^{n-1} \mathbf{G}_q^l(u_q)\Big) \prod_{q=p+1}^{n}M^l(u_{q-1},du_q),
$$
where we use the convention $\mathbf{Q}_{p,p}^l(\varphi)(u_p)=\varphi(u_p)$. 
One could interpret this as a type of weighted evolution operator from time $p$ to $n$ when a discretization of $\Delta_l$ is used.
In addition, for $(l,p,n)\in\mathbb{N}_0^3$, $n>p$, $(u_p,\varphi)\in E_l\times\mathcal{B}_b(E_l)$, let 
$$
\mathbf{D}_{p,n}^l(\varphi)(u_p) := \frac{\mathbf{Q}_{p,n}^l(\varphi-\pi_n^l(\varphi))(u_p)}{\pi_p^l(\mathbf{Q}_{p,n}^l(1))}
$$
where $\mathbf{D}_{p,p}^l(\varphi)(u_p)=\varphi(u_p)-\pi_p^l(\varphi)$. The latter operator would facilitate certain martingale decompositions in the following.

Throughout our arguments, $C$ is a finite constant whose value may change from line to line, but does not depend upon $l$ nor $N$. The particular dependencies of a given constant will be clear from the statement of a given result. 

As $h$ is bounded, there exists $-\infty<\underline{C}<\overline{C} <+\infty$, such that for any $(l,n)\in\mathbb{N}\times\mathbb{N}_0$ and $u_{n}\in E_l$, almost surely
\begin{eqnarray*}
\mathbf{G}_{n}^l(u_n) & \leq & \overline{\mathbf{G}}_{n}^l \\
\mathbf{G}_{n}^l(u_n) & \geq & \underline{\mathbf{G}}_{n}^l
\end{eqnarray*}
where
\begin{eqnarray}
\overline{\mathbf{G}}_{n}^l & = & \overline{C}\exp\Big\{\sum_{i=0}^{\Delta_l^{-1}-1}\sum_{k=1}^{d_y}\Big(\overline{C}\mathbb{I}_{[0,\infty)}((Y_{n+(i+1)\Delta_l}^{(k)}-Y_{n+i\Delta_l}^{(k)}))(Y_{n+(i+1)\Delta_l}^{(k)}-Y_{n+i\Delta_l}^{(k)})+ \nonumber\\ & &
\underline{C}\mathbb{I}_{(-\infty,0)}((Y_{n+(i+1)\Delta_l}^{(k)}-Y_{n+i\Delta_l}^{(k)}))(Y_{n+(i+1)\Delta_l}^{(k)}-Y_{n+i\Delta_l}^{(k)})
\Big)\Big\}\label{eq:g_upper}\\
\underline{\mathbf{G}}_{n}^l & = & \underline{C}\exp\Big\{\sum_{i=0}^{\Delta_l^{-1}-1}\sum_{k=1}^{d_y}\Big(\underline{C}\mathbb{I}_{[0,\infty)}((Y_{n+(i+1)\Delta_l}^{(k)}-Y_{n+i\Delta_l}^{(k)}))(Y_{n+(i+1)\Delta_l}^{(k)}-Y_{n+i\Delta_l}^{(k)})+ \nonumber \\ & &
\overline{C}\mathbb{I}_{(-\infty,0)}((Y_{n+(i+1)\Delta_l}^{(k)}-Y_{n+i\Delta_l}^{(k)}))(Y_{n+(i+1)\Delta_l}^{(k)}-Y_{n+i\Delta_l}^{(k)})
\Big)\Big\}\label{eq:g_lower}.
\end{eqnarray}
Moreover, for any $r\in\mathbb{N}$, it is straightforward to verify that these upper and lower bounds have finite $\mathbb{L}_r$ and $\mathbb{L}_{-r}$ moments that do not depend upon $l$. 

\subsection{Results for the Coupled Particle Filter}\label{app:cpf}

Set, for $l\in\mathbb{N}$, $(n,p,\varphi)\in\mathbb{N}_0^2\times\mathcal{B}_b(\mathbb{R}^{d_x})$, $p\leq n$
\begin{eqnarray*}
T_{p,n}^{l,1}(\varphi) & := & \overline{\mathbb{E}}[
\max\{(\underline{\mathbf{G}}_{n}^l)^{-2},(\underline{\mathbf{G}}_{n}^{l-1})^{-2}\}
(\mathbf{D}_{p,n}^l(\mathbf{G}_n^l\pmb{\varphi}^l)(U_p^{l,1})- \mathbf{D}_{p,n}^{l-1}(\mathbf{G}_n^{l-1}\pmb{\varphi}^{l-1})(\bar{U}_p^{l-1,1}))^2]
\end{eqnarray*}
and for $p<n$
\begin{eqnarray*}
T_{p,n}^{l,2}(\varphi) & := & \overline{\mathbb{E}}[(\mathbf{D}_{p,n}^l(\mathbf{G}_n^l\pmb{\varphi}^l)(U_p^{l,1})- \mathbf{D}_{p,n}^{l-1}(\mathbf{G}_n^{l-1}\pmb{\varphi}^{l-1})(\bar{U}_p^{l-1,1}))^2]^{1/2} + 
\|\varphi\|^2\overline{\mathbb{E}}[(\mathbf{G}_{p}^l(U_p^{l,1})-\mathbf{G}_{p}^{l-1}(\bar{U}_p^{l-1,1}))^2]^{1/2}.
\end{eqnarray*}

\begin{lem}\label{lem:l2_cpf}
Assume (D\ref{hyp_diff:1}). Then for any $n\in\mathbb{N}_0$, there exists a $C<+\infty$ such that for any $(l,N,\varphi)\in\mathbb{N}\times \mathbb{N}\times\mathcal{B}_b(\mathbb{R}^{d_x})$
$$
\overline{\mathbb{E}}\left[
\max\{\check{\pi}_n^{l,N}(\mathbf{G}_n^l\otimes 1)^{-2},\check{\pi}_n^{l,N}(1\otimes\mathbf{G}_n^{l-1})^{-2}\}
(\check{\pi}_n^{l,N}-\check{\pi}_n^l)((\mathbf{G}_n^l\pmb{\varphi}^l)\otimes 1 - 1\otimes(\mathbf{G}_n^{l-1}\pmb{\varphi}^{l-1}))^2\right] \leq 
$$
$$
C\left(
\frac{1}{N}\sum_{p=0}^nT_{p,n}^{l,1}(\varphi)
+ \frac{1}{N^{3/2}}\sum_{p=0}^{n-1}T_{p,n}^{l,2}(\varphi)
\right).
$$
\end{lem}

\begin{proof}
We have the following standard Martingale plus remainder decomposition \cite[Lemma 6.3]{ddj2012}
\begin{eqnarray*}
(\check{\pi}_n^{l,N}-\check{\pi}_n^l)((\mathbf{G}_n^l\pmb{\varphi}^l)\otimes 1 - 1\otimes(\mathbf{G}_n^{l-1}\pmb{\varphi}^{l-1})) & = & \sum_{p=0}^n(\check{\pi}_p^{l,N}-\check{\Phi}_p^l(\check{\pi}_{p-1}^{l,N}))(\mathbf{D}_{p,n}^l(\mathbf{G}_n^l\pmb{\varphi}^l)\otimes 1 - 1\otimes \mathbf{D}_{p,n}^{l-1}(\mathbf{G}_n^{l-1}\pmb{\varphi}^{l-1})) + \\ & &
 \sum_{p=0}^{n-1}\Big\{\frac{\check{\pi}_p^{l,N}(\mathbf{D}_{p,n}^l(\mathbf{G}_n^l\pmb{\varphi}^l)\otimes 1)}{\check{\pi}_p^{l,N}(\mathbf{G}_p^l\otimes 1)}[\check{\pi}_p^l-\check{\pi}_p^{l,N}](\mathbf{G}_p^l\otimes 1) - \\ & &
\frac{\check{\pi}_p^{l,N}(1\otimes\mathbf{D}_{p,n}^{l-1}(\mathbf{G}_n^{l-1}\pmb{\varphi}^{l-1}))}{\check{\pi}_p^{l,N}(1\otimes\mathbf{G}_p^{l-1})}[\check{\pi}_p^l-\check{\pi}_p^{l,N}](1\otimes\mathbf{G}_p^{l-1})
\Big\}.
\end{eqnarray*}
Using the $C_2-$inequality multiple times:
\begin{equation}\label{eq:tp}
\overline{\mathbb{E}}[
\max\{\check{\pi}_n^{l,N}(\mathbf{G}_n^l\otimes 1)^{-2},\check{\pi}_n^{l,N}(1\otimes\mathbf{G}_n^{l-1})^{-2}\}
(\check{\pi}_n^{l,N}-\check{\pi}_n^l)((\mathbf{G}_n^l\pmb{\varphi}^l)\otimes 1 - 1\otimes(\mathbf{G}_n^{l-1}\pmb{\varphi}^{l-1}))^2] \leq C\Big(\sum_{p=0}^n\overline{\mathbb{E}}[T_1(p)^2] + \sum_{p=0}^{n-1}\overline{\mathbb{E}}[T_2(p)^2]\Big)
\end{equation}
where
\begin{eqnarray*}
T_1(p) & := & 
\max\{\check{\pi}_pn^{l,N}(\mathbf{G}_n^l\otimes 1)^{-2},\check{\pi}_n^{l,N}(1\otimes\mathbf{G}_n^{l-1})^{-2}\}^{1/2}
(\check{\pi}_p^{l,N}-\check{\Phi}_p^l(\check{\pi}_{p-1}^{l,N}))(\mathbf{D}_{p,n}^l(\mathbf{G}_n^l\pmb{\varphi}^l)\otimes 1 - 1\otimes \mathbf{D}_{p,n}^{l-1}(\mathbf{G}_n^{l-1}\pmb{\varphi}^{l-1})) \\
T_2(p) & := & 
\max\{\check{\pi}_n^{l,N}(\mathbf{G}_n^l\otimes 1)^{-2},\check{\pi}_n^{l,N}(1\otimes\mathbf{G}_n^{l-1})^{-2}\}^{1/2}\times \\ & &
\left(\frac{\check{\pi}_p^{l,N}(\mathbf{D}_{p,n}^l(\mathbf{G}_n^l\pmb{\varphi}^l)\otimes 1)}{\check{\pi}_p^{l,N}(\mathbf{G}_p^l\otimes 1)}[\check{\pi}_p^l-\check{\pi}_p^{l,N}](\mathbf{G}_p^l\otimes 1) - 
\frac{\check{\pi}_p^{l,N}(1\otimes\mathbf{D}_{p,n}^{l-1}(\mathbf{G}_n^{l-1}\pmb{\varphi}^{l-1}))}{\check{\pi}_p^{l,N}(1\otimes\mathbf{G}_p^{l-1})}[\check{\pi}_p^l-\check{\pi}_p^{l,N}](1\otimes\mathbf{G}_p^{l-1})\right).
\end{eqnarray*}
It thus suffices to control the terms $T_1(p)$, $p\in\{0,1,\dots,n\}$ and  $T_2(p)$, $p\in\{0,1,\dots,n-1\}$ in an appropriate way. 

Now, using \eqref{eq:g_lower} and applying the conditional Marcinkiewicz-Zygmund inequality
\begin{equation}\label{eq:t1p}
\overline{\mathbb{E}}[T_1(p)^2] \leq \frac{C}{N}\overline{\mathbb{E}}[
\max\{(\underline{\mathbf{G}}_{n}^l)^{-2},(\underline{\mathbf{G}}_{n}^{l-1})^{-2}\}
(\mathbf{D}_{p,n}^l(\mathbf{G}_n^l\pmb{\varphi}^l)(U_p^{l,1})-\mathbf{D}_{p,n}^{l-1}(\mathbf{G}_n^{l-1}\pmb{\varphi}^{l-1})(\bar{U}_p^{l-1,1}))^2].
\end{equation}
For $T_2(p)$ we have
$$
T_2(p) = \max\{\check{\pi}_n^{l,N}(\mathbf{G}_n^l\otimes 1)^{-2},\check{\pi}_n^{l,N}(1\otimes\mathbf{G}_n^{l-1})^{-2}\}^{1/2}\left(T_3(p) + T_4(p) + T_5(p)\right)
$$
where
\begin{eqnarray}
T_3(p) & := & [\check{\pi}_p^l-\check{\pi}_p^{l,N}](\mathbf{G}_p^l\otimes 1)\frac{\check{\pi}_p^{l,N}(\mathbf{D}_{p,n}^l(\mathbf{G}_n^l\pmb{\varphi}^l)\otimes 1)}{\check{\pi}_p^{l,N}(\mathbf{G}_p^l\otimes 1)\check{\pi}_p^{l,N}(1\otimes\mathbf{G}_p^{l-1})}\Big\{\check{\pi}_p^{l,N}(1\otimes\mathbf{G}_p^{l-1})-\check{\pi}_p^{l,N}(\mathbf{G}_p^l\otimes 1)\Big\}\label{eq:t3p}\\
T_4(p) & := & [\check{\pi}_p^l-\check{\pi}_p^{l,N}](\mathbf{G}_p^l\otimes 1)\frac{1}{\check{\pi}_p^{l,N}(1\otimes\mathbf{G}_p^{l-1})}\Big\{\check{\pi}_p^{l,N}(\mathbf{D}_{p,n}^l(\mathbf{G}_n^l\pmb{\varphi}^l)\otimes 1)-\check{\pi}_p^{l,N}(1\otimes\mathbf{D}_{p,n}^{l-1}(\mathbf{G}_n^{l-1}\pmb{\varphi}^{l-1}))\Big\}\label{eq:t4p}\\
T_5(p) & := & \frac{\check{\pi}_p^{l,N}(1\otimes\mathbf{D}_{p,n}^{l-1}(\mathbf{G}_n^{l-1}\pmb{\varphi}^{l-1}))}{\check{\pi}_p^{l,N}(1\otimes\mathbf{G}_p^{l-1})}[\check{\pi}_p^l-\check{\pi}_p^{l,N}](\mathbf{G}_p^l\otimes 1-1\otimes\mathbf{G}_p^{l-1})\label{eq:t5p}.
\end{eqnarray}
By using \eqref{eq:g_lower} and the $C_2-$inequality three times 
$$
\overline{\mathbb{E}}[T_2(p)^2]\leq C\sum_{j=3}^5\overline{\mathbb{E}}[\max\{(\underline{\mathbf{G}}_{n}^l)^{-2},(\underline{\mathbf{G}}_{n}^{l-1})^{-2}\}T_j(p)^2]
$$
so we consider bounding the R.H.S.~of this inequality.

For $T_3(p)$, using Cauchy-Schwarz then H\"older gives
\begin{eqnarray*}
\overline{\mathbb{E}}[\max\{(\underline{\mathbf{G}}_{n}^l)^{-2},(\underline{\mathbf{G}}_{n}^{l-1})^{-2}\}T_3(p)^2] & \leq &  \overline{\mathbb{E}}\Big[\Big|\frac{\check{\pi}_p^{l,N}(\mathbf{D}_{p,n}^l(\mathbf{G}_n^l\pmb{\varphi}^l)\otimes 1)\max\{(\underline{\mathbf{G}}_{n}^l)^{-2},(\underline{\mathbf{G}}_{n}^{l-1})^{-2}\}}{\check{\pi}_p^{l,N}(\mathbf{G}_p^l\otimes 1)\check{\pi}_p^{l,N}(1\otimes\mathbf{G}_p^{l-1})}\Big|^{12}
\Big]^{1/6}\times \\
& & 
\overline{\mathbb{E}}[[\check{\pi}_p^l-\check{\pi}_p^{l,N}](\mathbf{G}_p^l\otimes 1)^{12}]^{1/6}
\overline{\mathbb{E}}[|\check{\pi}_p^{l,N}(1\otimes\mathbf{G}_p^{l-1})-\check{\pi}_p^{l,N}(\mathbf{G}_p^l\otimes 1)|^6]^{1/6}\times \\
& & 
\overline{\mathbb{E}}[|\check{\pi}_p^{l,N}(1\otimes\mathbf{G}_p^{l-1})-\check{\pi}_p^{l,N}(\mathbf{G}_p^l\otimes 1)|^2]^{1/2}.
\end{eqnarray*}
For the left-most term on the R.H.S.~one can apply H\"older, Lemma \ref{lem:res1} and Corollary \ref{cor:pf_cor}. 
For the term after on the R.H.S.~one can apply Proposition \ref{prop:lp_pf}.
For the next term, one can use \eqref{eq:g_upper}.
For the right-most term on the R.H.S.~one has 
$$
\overline{\mathbb{E}}[|\check{\pi}_p^{l,N}(1\otimes\mathbf{G}_p^{l-1})-\check{\pi}_p^{l,N}(\mathbf{G}_p^l\otimes 1)|^2]^{1/2}\leq \overline{\mathbb{E}}[(\mathbf{G}_{p}^l(U_p^{l,1})-\mathbf{G}_{p}^{l-1}(\bar{U}_p^{l-1,1}))^2]^{1/2}.
$$

Hence, we have that
\begin{equation}\label{eq:t3pp}
\overline{\mathbb{E}}[T_3(p)^2] \leq \frac{C\|\varphi\|^2}{N^{2}}\overline{\mathbb{E}}[(\mathbf{G}_{p}^l(U_p^{l,1})-\mathbf{G}_{p}^{l-1}(\bar{U}_p^{l-1,1}))^2]^{1/2}.
\end{equation}
For $T_4(p), T_5(p)$, using almost the same strategy , except for using Proposition \ref{prop:lp_pf} for terms such as (for any $r\in\mathbb{N}$)
$$
\mathbb{E}\left[\Big|\check{\pi}_p^{l,N}(\mathbf{D}_{p,n}^l(\mathbf{G}_n^l\pmb{\varphi}^l)\otimes 1)-\check{\pi}_p^{l,N}(1\otimes\mathbf{D}_{p,n}^{l-1}(\mathbf{G}_n^{l-1}\pmb{\varphi}^{l-1}))\Big|^r\right]^{1/r}
$$
yields
\begin{equation}\label{eq:t4pp}
\overline{\mathbb{E}}[\max\{(\underline{\mathbf{G}}_{n}^l)^{-2},(\underline{\mathbf{G}}_{n}^{l-1})^{-2}\}T_4(p)^2] \leq \frac{C}{N^{3/2}}\overline{\mathbb{E}}[(\mathbf{D}_{p,n}^l(\mathbf{G}_n^l\pmb{\varphi}^l)(U_p^{l,1})- \mathbf{D}_{p,n}^{l-1}(\mathbf{G}_n^{l-1}\pmb{\varphi}^{l-1})(\bar{U}_p^{l-1,1}))^2]^{1/2}.
\end{equation}
and 
\begin{equation}\label{eq:t5pp}
\overline{\mathbb{E}}[\max\{(\underline{\mathbf{G}}_{n}^l)^{-2},(\underline{\mathbf{G}}_{n}^{l-1})^{-2}\}T_5(p)^2] \leq \frac{C\|\varphi\|^2}{N^{3/2}}\overline{\mathbb{E}}[(\mathbf{G}_{p}^l(U_p^{l,1})-\mathbf{G}_{p}^{l-1}(\bar{U}_p^{l-1,1}))^2]^{1/2}.
\end{equation}

Combining \eqref{eq:t3pp}-\eqref{eq:t5pp} gives
\begin{equation}\label{eq:t2p}
\overline{\mathbb{E}}[T_2(p)^2] \leq \frac{C}{N^{3/2}}T_{p,n}^{l,2}(\varphi).
\end{equation}
The proof is easily completed by noting the bounds in \eqref{eq:tp}, \eqref{eq:t1p} and \eqref{eq:t2p}.
\end{proof}

\subsubsection{Additional Technical Results for Coupled Particle Filters}\label{app:cpf1}

The following section is essentially an adaptation of \cite[Lemmata D.3-D.4]{mlpf}. Many of the arguments are very similar to
that article, with a modification to the context here. The entire proofs are included for completeness of this paper.

For $(i,l,n)\in\{1,\dots,N\}\times \mathbb{N}\times\mathbb{N}_0$:
\begin{itemize}
\item{$\widehat{U}_{n}^{l,i},\widehat{\bar{U}}_{n}^{l-1,i}$ denote the particles immediately after resampling}
\item{$(I_{n}^{l,i},\bar{I}_{n}^{l-1,i})\in\{1,\dots,N\}$ represent the resampled indices of $(u_n^{l,i},\bar{u}_n^{l-1,i})$ and
let $I_{n}^l(i):=I_{n}^{l,i}$ and  $\bar{I}_{n}^{l-1}(i):=\bar{I}_{n}^{l-1,i}$.}
\end{itemize}
For $(l,n)\in\mathbb{N}\times\mathbb{N}_0$, let $\mathsf{S}_n^l$ be the particle indices that choose the same ancestor at each resampling stage: 
	\begin{align*}
	\mathsf{S}_n^l=\{i\in\{1,\ldots, N\}: &I_{n}^l(i)=\bar{I}_{n}^{l-1}(i),I_{n-1}^l\circ I_{n}^l(i)=\bar{I}_{n-1}^{l-1}\circ \bar{I}_{n}^{l-1}(i),\dots,
	I_{0}^l\circ\cdots\circ I_{n}^l(i)=\bar{I}_{0}^{l-1}\circ\cdots\circ \bar{I}_{n}^{l-1}(i)\}. 
	\end{align*}
	For $n=-1$, set $\mathsf{S}_n^l=\{1,\ldots, N\}$. 
	Let, for $(l,n)\in\mathbb{N}\times\mathbb{N}_0$
\begin{align*}
\mathcal{G}_n^l=&\sigma\left(\left\{U_{p}^{l,i},\bar{U}_{p}^{l-1,i}, \widehat{U}_{p}^{l,i}, \widehat{\bar{U}}_{p}^{l-1,i}, I_{p}^l, \bar{I}_{p}^{l-1};0\leq p<n,1\leq i\le N\right\}\cup
\left\{U_{n}^{l,i},\bar{U}_{n}^{l-1,i},1\leq i\le N\right\}\right)\vee\mathcal{Y}_{n+1},\\
\widehat{\mathcal{G}}_n^l=&\sigma\left(\left\{U_{p}^{l,i},\bar{U}_{p}^{l-1,i}, \widehat{U}_{p}^{l,i}, \widehat{\bar{U}}_{p}^{l,i}, I_{p}^l, \bar{I}_{p}^{l-1};0\leq p<n,1\leq i\le N\right\}\cup
\left\{U_{n}^{l,i},\bar{U}_{n}^{l-1,i},\widehat{U}_{n}^{l,i},\widehat{\bar{U}}_{n}^{l-1,i},1\leq i\le N\right\}\right)\vee\mathcal{Y}_{n+1}. 
\end{align*}
To avoid ambiguity in the subsequent notations, we set for $(i,l,n)\in\{1,\dots,N\}\times \mathbb{N}\times\mathbb{N}_0$
\begin{eqnarray*}
u_{n}^{l,i} & = & (x_{n,n}^{l,i},x_{n,n+\Delta_l}^{l,i},\dots,x_{n,n+1}^{l,i}) \in E_l \\
\bar{u}_{n}^{l,i} & = & (\bar{x}_{n,n}^{l-1,i},\bar{x}_{n,n+\Delta_{l-1}}^{l-1,i},\dots,\bar{x}_{n,n+1}^{l-1,i}) \in E_{l-1}.
\end{eqnarray*}

\begin{lem}\label{lem:pairwise_difference}
Assume (D\ref{hyp_diff:1}). Then for any $(n,r)\in\mathbb{N}_0\times\mathbb{N}$, there exists a $C<+\infty$ such that for any $(l,N)\in\mathbb{N}\times\mathbb{N}$
\begin{align*} 
	\overline{\mathbb{E}}\left[\frac{1}{N}\sum_{i\in \mathsf{S}_{n-1}^l} \|X_{n,n+1}^{l,i}-\bar{X}_{n,n+1}^{l-1,i}\|_2^r\right]^{1/r}\le C\Delta_{l}^{1/2}.
\end{align*}
\end{lem}

\begin{proof}
The proof is by induction on $n$. The case $n=0$ follows immediately, for instance by \cite[Proposition D.1]{mlpf}. For a general $n$; following the first four lines of the proof
of \cite[Lemma D.3]{mlpf}, one has
$$
\overline{\mathbb{E}}\left[\frac{1}{N}\sum_{i\in \mathsf{S}_{n-1}^l} \|X_{n,n+1}^{l,i}-\bar{X}_{n,n+1}^{l-1,i}\|_2^r\right]^{1/r}\leq
C\overline{\mathbb{E}}\left[\frac{1}{N}\sum_{i\in \mathsf{S}_{n-1}^l} \|X_{n-1,n}^{l,I_{n-1}^{l,i}}-\bar{X}_{n-1,n}^{l-1,\bar{I}_{n-1}^{l,i}}\|_2^r\right]^{1/r}+C\Delta_{l}^{1/2}.
$$
Now, $I_{n-1}^{l,i}=\bar{I}_{n-1}^{l-1,i}$  for $i\in\mathsf{S}_{n-1}^l$. 
The conditional distribution of $(X_{n-1,n}^{l,I_{n-1}^{l,i}},\bar{X}_{n-1,n}^{l,\bar{I}_{n-1,n}^{l,i}})\ (i\in \mathsf{S}_{n-1}^l)$ given $\mathsf{S}_{n-1}^l$  and $\mathcal{G}_{n-1}^l$ is
$$
\frac{\sum_{i\in \mathsf{S}_{n-2}^l} \frac{\mathbf{G}_{n-1}^l(U_{n-1}^{l,i})}{\sum_{k=1}^{N}\mathbf{G}_{n-1}^l(U_{n-1}^{l,k})}\wedge
\frac{\mathbf{G}_{n-1}^{l-1}(\bar{U}_{n-1}^{l-1,i})}{\sum_{k=1}^{N}\mathbf{G}_{n-1}^{l-1}(\bar{U}_{n-1}^{l-1,k})}\delta_{(X_{n-1,n}^{l,i},\bar{X}_{n-1,n}^{l,i})}}
{\sum_{i\in \mathsf{S}_{n-2}^l} \frac{\mathbf{G}_{n-1}^l(U_{n-1}^{l,i})}{\sum_{k=1}^{N}\mathbf{G}_{n-1}^l(U_{n-1}^{l,k})}\wedge
\frac{\mathbf{G}_{n-1}^{l-1}(\bar{U}_{n-1}^{l-1,i})}{\sum_{k=1}^{N}\mathbf{G}_{n-1}^{l-1}(\bar{U}_{n-1}^{l-1,k})}}.
$$
Now, we have, almost surely:
\begin{eqnarray}
\overline{\mathbb{E}}\left[\left.\frac{\textrm{Card}(\mathsf{S}_{n-1}^l)}{N}\right|\mathcal{G}_{n-1}^l\right] 
& = & \sum_{i\in \mathsf{S}_{n-2}^l} \frac{\mathbf{G}_{n-1}^l(U_{n-1}^{l,i})}{\sum_{k=1}^{N}\mathbf{G}_{n-1}^l(U_{n-1}^{l,k})}\wedge
\frac{\mathbf{G}_{n-1}^{l-1}(\bar{U}_{n-1}^{l-1,i})}{\sum_{k=1}^{N}\mathbf{G}_{n-1}^{l-1}(\bar{U}_{n-1}^{l-1,k})}\nonumber\\
& \leq & \frac{\overline{\mathbf{G}}_{n-1}^l}{\underline{\mathbf{G}}_{n-1}^l}\frac{\textrm{Card}(\mathsf{S}_{n-2}^l)}{N}.\label{eq:mistake1} 
\end{eqnarray}

Therefore
\begin{align*}
&\overline{\mathbb{E}}\left[\frac{1}{N}\sum_{i\in \mathsf{S}_{n-1}^l} \|X_{n-1,n}^{l,I_{n-1}^{l,i}}-\bar{X}_{n-1,n}^{l-1,\bar{I}_{n-1}^{l,i}}\|_2^r\right]\\
&=\overline{\mathbb{E}}\left[\frac{1}{N}\sum_{i\in \mathsf{S}_{n-1}^l} \overline{\mathbb{E}}\left[\left.\|X_{n-1,n}^{l,I_{n-1}^{l,i}}-\bar{X}_{n-1,n}^{l-1,\bar{I}_{n-1}^{l-1,i}}\|_2^r\right|\mathsf{S}_{n-1}^l,\mathcal{G}_{n-1}^l\right]\right]\\
&=\overline{\mathbb{E}}\left[\frac{\textrm{Card}(\mathsf{S}_{n-1}^l)}{N}\left\{
\frac{\sum_{i\in \mathsf{S}_{n-2}^l} \frac{\mathbf{G}_{n-1}^l(U_{n-1}^{l,i})}{\sum_{k=1}^{N}\mathbf{G}_{n-1}^l(U_{n-1}^{l,k})}\wedge
\frac{\mathbf{G}_{n-1}^{l-1}(\bar{U}_{n-1}^{l-1,i})}{\sum_{k=1}^{N}\mathbf{G}_{n-1}^{l-1}(\bar{U}_{n-1}^{l-1,k})}\|X_{n-1,n}^{l,i}-\bar{X}_{n-1,n}^{l,i}\|_2^r}
{\sum_{i\in \mathsf{S}_{n-2}^l} \frac{\mathbf{G}_{n-1}^l(U_{n-1}^{l,i})}{\sum_{k=1}^{N}\mathbf{G}_{n-1}^l(U_{n-1}^{l,k})}\wedge
\frac{\mathbf{G}_{n-1}^{l-1}(\bar{U}_{n-1}^{l-1,i})}{\sum_{k=1}^{N}\mathbf{G}_{n-1}^{l-1}(\bar{U}_{n-1}^{l-1,k})}}
\right\}\right]\\
&=\overline{\mathbb{E}}\left[\overline{\mathbb{E}}\left[\left.\frac{\textrm{Card}(\mathsf{S}_{n-1}^l)}{N}\right|\mathcal{G}_{n-1}^l\right]\left\{
\frac{\sum_{i\in \mathsf{S}_{n-2}^l} \frac{\mathbf{G}_{n-1}^l(U_{n-1}^{l,i})}{\sum_{k=1}^{N}\mathbf{G}_{n-1}^l(U_{n-1}^{l,k})}\wedge
\frac{\mathbf{G}_{n-1}^{l-1}(\bar{U}_{n-1}^{l-1,i})}{\sum_{k=1}^{N}\mathbf{G}_{n-1}^{l-1}(\bar{U}_{n-1}^{l-1,k})}\|X_{n-1,n}^{l,i}-\bar{X}_{n-1,n}^{l,i}\|_2^r}
{\sum_{i\in \mathsf{S}_{n-2}^l} \frac{\mathbf{G}_{n-1}^l(U_{n-1}^{l,i})}{\sum_{k=1}^{N}\mathbf{G}_{n-1}^l(U_{n-1}^{l,k})}\wedge
\frac{\mathbf{G}_{n-1}^{l-1}(\bar{U}_{n-1}^{l-1,i})}{\sum_{k=1}^{N}\mathbf{G}_{n-1}^{l-1}(\bar{U}_{n-1}^{l-1,k})}}
\right\}\right]\\
&\le \overline{\mathbb{E}}\left[\overline{\mathbb{E}}\left[\left.\frac{\textrm{Card}(\mathsf{S}_{n-1}^l)}{N}\right|\mathcal{G}_{n-1}^l\right]
\frac{1}{\textrm{Card}(\mathsf{S}_{n-2}^l)}\sum_{i\in \mathsf{S}_{n-2}^l} \|X_{n-1,n}^{l,i}-\bar{X}_{n-1,n}^{l-1,i}\|^r
\Bigg(\frac{\overline{\mathbf{G}}_{n-1}^l}{\underline{\mathbf{G}}_{n-1}^l}\wedge \frac{\overline{\mathbf{G}}_{n-1}^{l-1}}{\underline{\mathbf{G}}_{n-1}^{l-1}}\Big/
\frac{\underline{\mathbf{G}}_{n-1}^l}{\overline{\mathbf{G}}_{n-1}^l}\wedge \frac{\underline{\mathbf{G}}_{n-1}^{l-1}}{\overline{\mathbf{G}}_{n-1}^{l-1}}
\Bigg)
\right].
\end{align*}
Then noting \eqref{eq:mistake1}  and then taking expectations w.r.t.~the data on the time interval $[n-1,n]$ yields:
$$
\overline{\mathbb{E}}\left[\frac{1}{N}\sum_{i\in \mathsf{S}_{n-1}^l} \|X_{n-1,n}^{l,I_{n-1}^{l,i}}-\bar{X}_{n-1,n}^{l-1,\bar{I}_{n-1}^{l,i}}\|_2^r\right]
\le C\overline{\mathbb{E}}\left[\frac{1}{N}\sum_{i\in \mathsf{S}_{n-2}^l} \|X_{n-1,n}^{l,i}-\bar{X}_{n-1,n}^{l-1,i}\|^r\right].
$$
The result hence follows by induction. 
\end{proof}

\begin{cor}\label{cor:pairwise_difference}
Assume (D\ref{hyp_diff:1}). Then for any $(n,r)\in\mathbb{N}\times\mathbb{N}$, there exists a $C<+\infty$ such that for any $(l,N)\in\mathbb{N}\times\mathbb{N}$
\begin{align*} 
	\overline{\mathbb{E}}\left[\frac{1}{N}\sum_{i\in \mathsf{S}_{n-1}^l} \|\widehat{X}_{n-1,n}^{l,i}-\widehat{\bar{X}}_{n-1,n}^{l-1,i}\|_2^r\right]^{1/r}\le C\Delta_{l}^{1/2}.
\end{align*}
\end{cor}

\begin{proof}
Easily follows from the proof of Lemma \ref{lem:pairwise_difference}.
\end{proof}

\begin{lem}\label{small_probability}
Assume (D\ref{hyp_diff:1}). Then for any $n\in\{-1,0,1,\dots\}$, there exists a $C<+\infty$ such that for any $(l,N)\in\mathbb{N}\times\mathbb{N}$
\begin{align*}
1-\overline{\mathbb{E}}\left[\frac{\textrm{\emph{Card}}(\mathsf{S}_n^l)}{N}\right]\le C\Delta_{l}^{1/2}. 
\end{align*}
\end{lem}

\begin{proof}
The proof is by induction, with the initialization clear. We have 
\begin{align}
1-\overline{\mathbb{E}}\left[\frac{\textrm{Card}(\mathsf{S}_n^l)}{N}\right]
	=&\left\{1-\overline{\mathbb{E}}\left[\sum_{i=1}^N \frac{\mathbf{G}_{n}^l(U_{n}^{l,i})}{\sum_{k=1}^{N}\mathbf{G}_{n}^l(U_{n}^{l,k})}\wedge
\frac{\mathbf{G}_{n}^{l-1}(\bar{U}_{n}^{l-1,i})}{\sum_{k=1}^{N}\mathbf{G}_{n}^{l-1}(\bar{U}_{n}^{l-1,k})}\right]\right\}+\nonumber\\
&\overline{\mathbb{E}}\left[\sum_{i\notin\mathsf{S}_{n-1}^l}\frac{\mathbf{G}_{n}^l(U_{n}^{l,i})}{\sum_{k=1}^{N}\mathbf{G}_{n}^l(U_{n}^{l,k})}-
\frac{\mathbf{G}_{n}^{l-1}(\bar{U}_{n}^{l-1,i})}{\sum_{k=1}^{N}\mathbf{G}_{n}^{l-1}(\bar{U}_{n}^{l-1,k})}\right]\nonumber\\
&\leq C\overline{\mathbb{E}}\left[\sum_{i\in\mathsf{S}_{n-1}^l}\left|\frac{\mathbf{G}_{n}^l(U_{n}^{l,i})}{\sum_{k=1}^{N}\mathbf{G}_{n}^l(U_{n}^{l,k})}-
\frac{\mathbf{G}_{n}^{l-1}(\bar{U}_{n}^{l-1,i})}{\sum_{k=1}^{N}\mathbf{G}_{n}^{l-1}(\bar{U}_{n}^{l-1,k})}\right|\right] \nonumber\\
&+ C\overline{\mathbb{E}}\left[\sum_{i\notin\mathsf{S}_{n-1}^l}\left|\frac{\mathbf{G}_{n}^l(U_{n}^{l,i})}{\sum_{k=1}^{N}\mathbf{G}_{n}^l(U_{n}^{l,k})}-
\frac{\mathbf{G}_{n}^{l-1}(\bar{U}_{n}^{l-1,i})}{\sum_{k=1}^{N}\mathbf{G}_{n}^{l-1}(\bar{U}_{n}^{l-1,k})}\right|\right]\label{eq:prob_bound}.
\end{align}
Now, we note that by using the bounds from \eqref{eq:g_upper} and \eqref{eq:g_lower}
\begin{align}
\overline{\mathbb{E}}\left[\sum_{i\notin\mathsf{S}_{n-1}^l} \left|\frac{\mathbf{G}_{n}^l(U_{n}^{l,i})}{\sum_{k=1}^{N}\mathbf{G}_{n}^l(U_{n}^{l,k})}-
\frac{\mathbf{G}_{n}^{l-1}(\bar{U}_{n}^{l-1,i})}{\sum_{k=1}^{N}\mathbf{G}_{n}^{l-1}(\bar{U}_{n}^{l-1,k})}\right|\right] &\leq \overline{\mathbb{E}}\left[\frac{\textrm{Card}((\mathsf{S}_{n-1}^l)^c)}{N}
\Big(\frac{\overline{\mathbf{G}}_{n}^l}{\underline{\mathbf{G}}_{n}^l}+ \frac{\overline{\mathbf{G}}_{n}^{l-1}}{\underline{\mathbf{G}}_{n}^{l-1}}\Big)\right] \nonumber\\
&\leq C\Big(1-\overline{\mathbb{E}}\left[\frac{\textrm{Card}(\mathsf{S}_{n-1}^l)}{N}\right]\Big).\label{eq:small_prob1}
\end{align}

To conclude the result, we must appropriately deal with the left-most term on the R.H.S.~of \eqref{eq:prob_bound}. We have
$$
\overline{\mathbb{E}}\left[\sum_{i\in\mathsf{S}_{n-1}^l}\left|\frac{\mathbf{G}_{n}^l(U_{n}^{l,i})}{\sum_{k=1}^{N}\mathbf{G}_{n}^l(U_{n}^{l,k})}-
\frac{\mathbf{G}_{n}^{l-1}(\bar{U}_{n}^{l-1,i})}{\sum_{k=1}^{N}\mathbf{G}_{n}^{l-1}(\bar{U}_{n}^{l-1,k})}\right|\right] \leq T_1 + T_2
$$
where
\begin{eqnarray*}
T_1 & := & \overline{\mathbb{E}}\left[\frac{1}{\sum_{k=1}^{N}\mathbf{G}_{n}^l(U_{n}^{l,k})}\sum_{i\in\mathsf{S}_{n-1}^l}\left|\mathbf{G}_{n}^l(U_{n}^{l,i})-\mathbf{G}_{n}^{l-1}(\bar{U}_{n}^{l-1,i})\right|\right]\\
T_2 & := &  \overline{\mathbb{E}}\left[\sum_{i\in\mathsf{S}_{n-1}^l}\mathbf{G}_{n}^l(\bar{U}_{n}^{l-1,i})\left(
\frac{\sum_{k=1}^{N}\mathbf{G}_{n}^{l-1}(\bar{U}_{n}^{l-1,k})-\sum_{k=1}^{N}\mathbf{G}_{n}^l(U_{n}^{l,k})}{\sum_{k=1}^{N}\mathbf{G}_{n}^{l-1}(\bar{U}_{n}^{l-1,k})\sum_{k=1}^{N}\mathbf{G}_{n}^l(U_{n}^{l,k})}
\right)\right].
\end{eqnarray*}
For $T_1$, applying Cauchy-Schwarz and recalling the bounds from \eqref{eq:g_upper} and \eqref{eq:g_lower},
$$
T_1 \leq  \overline{\mathbb{E}}\left[\frac{1}{(\underline{\mathbf{G}}_{n}^l)^2}\right]^{1/2}\overline{\mathbb{E}}\left[\left(\frac{1}{N}\sum_{i\in\mathsf{S}_{n-1}^l}\left|\mathbf{G}_{n}^l(U_{n}^{l,i})-\mathbf{G}_{n}^{l-1}(\bar{U}_{n}^{l-1,i})\right|\right)^2\right]^{1/2}.
$$
Then applying Jensen and noting that the left-most term on the R.H.S.~is upper-bounded by a constant that does not depend upon $l$ nor $N$ we have
$$
T_1 \leq  C\overline{\mathbb{E}}\left[\frac{\textrm{Card}(\mathsf{S}_{n-1}^l)}{N}\frac{1}{N}\sum_{i\in\mathsf{S}_{n-1}^l}\left|\mathbf{G}_{n}^l(U_{n}^{l,i})-\mathbf{G}_{n}^{l-1}(\bar{U}_{n}^{l-1,i})\right|^2\right]^{1/2}.
$$
Conditioning upon $\widehat{\mathcal{G}}_{n-1}^l$ and applying Lemma \ref{lem:main_nlf_lem} followed by Corollary \ref{cor:pairwise_difference} yields the upper-bound
$$
T_1 \leq C\Delta_l^{1/2}.
$$
For $T_2$, using the bounds \eqref{eq:g_upper} and \eqref{eq:g_lower}, one has
$$
T_2 \leq \overline{\mathbb{E}}\left[\frac{\textrm{Card}(\mathsf{S}_{n-1}^l)}{N}\left(\frac{\overline{\mathbf{G}}_{n}^l}{\underline{\mathbf{G}}_{n}^l\underline{\mathbf{G}}_{n}^{l-1}}\right)\left|\frac{1}{N}\sum_{i=1}^N\{\mathbf{G}_{n}^l(U_{n}^{l,i})-\mathbf{G}_{n}^{l-1}(\bar{U}_{n}^{l-1,i})\}\right|\right].
$$
Then it easily follows
$$
T_2 \leq \overline{\mathbb{E}}\left[\frac{\textrm{Card}(\mathsf{S}_{n-1}^l)}{N}\left(\frac{\overline{\mathbf{G}}_{n}^l}{\underline{\mathbf{G}}_{n}^l\underline{\mathbf{G}}_{n}^{l-1}}\right)
\left(\left|\frac{1}{N}\sum_{i\in \mathsf{S}_{n-1}^l}\{\mathbf{G}_{n}^l(U_{n}^{l,i})-\mathbf{G}_{n}^{l-1}(\bar{U}_{n}^{l-1,i})\}\right| + \frac{\textrm{Card}((\mathsf{S}_{n-1}^l)^c)}{N}(\overline{\mathbf{G}}_{n}^l+\overline{\mathbf{G}}_{n}^{l-1})
\right)\right].
$$
Now we set
\begin{eqnarray*}
T_3 & := & \overline{\mathbb{E}}\left[\frac{\textrm{Card}(\mathsf{S}_{n-1}^l)}{N}\left(\frac{\overline{\mathbf{G}}_{n}^l}{\underline{\mathbf{G}}_{n}^l\underline{\mathbf{G}}_{n}^{l-1}}\right)
\left|\frac{1}{N}\sum_{i\in \mathsf{S}_{n-1}^l}\{\mathbf{G}_{n}^l(U_{n}^{l,i})-\mathbf{G}_{n}^{l-1}(\bar{U}_{n}^{l-1,i})\}\right|\right] \\
T_4 & := & \overline{\mathbb{E}}\left[\frac{\textrm{Card}(\mathsf{S}_{n-1}^l)}{N}\left(\frac{\overline{\mathbf{G}}_{n}^l}{\underline{\mathbf{G}}_{n}^l\underline{\mathbf{G}}_{n}^{l-1}}\right)
\frac{\textrm{Card}((\mathsf{S}_{n-1}^l)^c)}{N}(\overline{\mathbf{G}}_{n}^l+\overline{\mathbf{G}}_{n}^{l-1})\right].
\end{eqnarray*}
For $T_3$, applying Cauchy Schwarz and Jensen
$$
T_3 \leq \overline{\mathbb{E}}\left[\left(\frac{\textrm{Card}(\mathsf{S}_{n-1}^l)}{N}\left(\frac{\overline{\mathbf{G}}_{n}^l}{\underline{\mathbf{G}}_{n}^l\underline{\mathbf{G}}_{n}^{l-1}}\right)\right)^2\right]^{1/2}\overline{\mathbb{E}}\left[\frac{1}{N}\sum_{i\in \mathsf{S}_{n-1}^l}\{\mathbf{G}_{n}^l(U_{n}^{l,i})-\mathbf{G}_{n}^{l-1}(\bar{U}_{n}^{l-1,i})\}^2\right]^{1/2}.
$$
Again, noting that the left-most term on the R.H.S.~is upper-bounded by a constant that does not depend upon $l$ nor $N$ 
$$
T_3 \leq C \overline{\mathbb{E}}\left[\frac{1}{N}\sum_{i\in S_{n-1}^l}\{\mathbf{G}_{n}^l(U_{n}^{l,i})-\mathbf{G}_{n}^{l-1}(\bar{U}_{n}^{l-1,i})\}^2\right]^{1/2}.
$$
Then, applying the above arguments, we have
$$
T_3 \leq  C\Delta_l^{1/2}.
$$
For $T_4$, taking expectations w.r.t.~the data on the time interval $[n,n+1]$ yields:
$$
T_4 \leq C\overline{\mathbb{E}}\left[\frac{\textrm{Card}((\mathsf{S}_{n-1}^l)^c)}{N}\right].
$$
Thus, we have
\begin{equation}\label{eq:small_prob2}
\overline{\mathbb{E}}\left[\sum_{i\in\mathsf{S}_{n-1}^l}\left|\frac{\mathbf{G}_{n}^l(U_{n}^{l,i})}{\sum_{k=1}^{N}\mathbf{G}_{n}^l(U_{n}^{l,k})}-
\frac{\mathbf{G}_{n}^{l-1}(\bar{U}_{n}^{l-1,i})}{\sum_{k=1}^{N}\mathbf{G}_{n}^{l-1}(\bar{U}_{n}^{l-1,k})}\right|\right] \leq C\Big(\Delta_l^{1/2}+1-\overline{\mathbb{E}}\left[\frac{\textrm{Card}(\mathsf{S}_{n-1}^l)}{N}\right]\Big).
\end{equation}
Combining \eqref{eq:prob_bound}, \eqref{eq:small_prob1}, \eqref{eq:small_prob2}, one can conclude the result via induction.
\end{proof}

\subsubsection{Rate Proofs for the Coupled Particle Filter}\label{app:cpf22}

\begin{lem}\label{lem:l2_cpf_contr}
Assume (D\ref{hyp_diff:1}). Then for any $n\in\mathbb{N}_0$, there exists a $C<+\infty$ such that for any $(l,N,\varphi)\in\mathbb{N}\times \mathbb{N}\times\mathcal{B}_b(\mathbb{R}^{d_x})\cap\textrm{\emph{Lip}}_{\|\cdot\|_2}(\mathbb{R}^{d_x})$
\begin{eqnarray*}
\sum_{p=0}^nT_{p,n}^{l,1}(\varphi) & \leq &  C(\|\varphi\|+\|\varphi\|_{\textrm{\emph{Lip}}})^2\Delta_l^{1/2} \\
\sum_{p=0}^{n-1}T_{p,n}^{l,2}(\varphi) & \leq &  C(\|\varphi\|+\|\varphi\|_{\textrm{\emph{Lip}}})^2\Delta_l^{1/4}.
\end{eqnarray*}
\end{lem}

\begin{proof}
We consider $T_{p,n}^{l,1}(\varphi)$ only, as the case of $T_{p,n}^{l,2}(\varphi)$ is very similar. 

We have the upper-bound
$$
\overline{\mathbb{E}}[\max\{(\underline{\mathbf{G}}_{n}^l)^{-2},(\underline{\mathbf{G}}_{n}^{l-1})^{-2}\}(\mathbf{D}_{p,n}^l(\mathbf{G}_n^l\pmb{\varphi}^l)(U_p^{l,1})- \mathbf{D}_{p,n}^{l-1}(\mathbf{G}_n^{l-1}\pmb{\varphi}^{l-1})(\bar{U}_p^{l-1,1}))^2] \leq C\sum_{j=1}^5T_j
$$
where
\begin{eqnarray*}
T_1 & := & \overline{\mathbb{E}}\left[
\max\{(\underline{\mathbf{G}}_{n}^l)^{-2},(\underline{\mathbf{G}}_{n}^{l-1})^{-2}\}
\left(\frac{1}{\pi_p^l(\mathbf{Q}_{p,n}^l(1))}\{\mathbf{Q}_{p,n}^l(\mathbf{G}_n^l\pmb{\varphi}^l)(U_p^{l,1})- \mathbf{Q}_{p,n}^{l-1}(\mathbf{G}_n^{l-1}\pmb{\varphi}^{l-1})(\bar{U}_p^{l-1,1})\}\right)^2\right]\\
T_2 & := & \overline{\mathbb{E}}\left[
\max\{(\underline{\mathbf{G}}_{n}^l)^{-2},(\underline{\mathbf{G}}_{n}^{l-1})^{-2}\}
\left(\frac{\mathbf{Q}_{p,n}^{l-1}(\mathbf{G}_n^{l-1}\pmb{\varphi}^{l-1})(\bar{U}_p^{l-1,1})}{\pi_p^l(\mathbf{Q}_{p,n}^l(1))\pi_p^{l-1}(\mathbf{Q}_{p,n}^{l-1}(1))}
\{\pi_p^{l-1}(\mathbf{Q}_{p,n}^{l-1}(1))-\pi_p^l(\mathbf{Q}_{p,n}^l(1))\}\right)^2\right]\\
T_3 & := & \overline{\mathbb{E}}\left[
\max\{(\underline{\mathbf{G}}_{n}^l)^{-2},(\underline{\mathbf{G}}_{n}^{l-1})^{-2}\}
\left(\frac{\mathbf{Q}_{p,n}^{l-1}(1)(\bar{U}_p^{l-1,1})}{\pi_p^{l-1}(\mathbf{Q}_{p,n}^{l-1}(1))}
\{\pi_n^{l-1}(\mathbf{G}_n^{l-1}\pmb{\varphi}^{l-1})-\pi_n^l(\mathbf{G}_n^{l}\pmb{\varphi}^{l})\}\right)^2\right]\\
T_4 & := & \overline{\mathbb{E}}\left[
\max\{(\underline{\mathbf{G}}_{n}^l)^{-2},(\underline{\mathbf{G}}_{n}^{l-1})^{-2}\}
\left(\frac{\pi_n^l(\mathbf{G}_n^{l}\pmb{\varphi}^{l})}{\pi_p^{l-1}(\mathbf{Q}_{p,n}^{l-1}(1))}\{\mathbf{Q}_{p,n}^{l-1}(1)(\bar{U}_p^{l-1,1})-\mathbf{Q}_{p,n}^l(1)(U_p^{l,1})\}\right)^2\right]\\
T_5 & := & \overline{\mathbb{E}}\left[
\max\{(\underline{\mathbf{G}}_{n}^l)^{-2},(\underline{\mathbf{G}}_{n}^{l-1})^{-2}\}
\left(\frac{\mathbf{Q}_{p,n}^{l}(1)(U_p^{l,1})}{\pi_p^l(\mathbf{Q}_{p,n}^l(1))\pi_p^{l-1}(\mathbf{Q}_{p,n}^{l-1}(1))}
\{\pi_p^l(\mathbf{Q}_{p,n}^l(1))-\pi_p^{l-1}(\mathbf{Q}_{p,n}^{l-1}(1))\}\right)^2\right].
\end{eqnarray*}
We will give bounds on the terms $T_1$ and $T_2$ only. The proofs for appropriate bounds on $T_3, T_4, T_5$ are very similar and hence omitted.

For $T_1$, we have $T_1=T_6+T_7$, where
\begin{eqnarray*}
T_6 &:=& \overline{\mathbb{E}}\left[\frac{\max\{(\underline{\mathbf{G}}_{n}^l)^{-2},(\underline{\mathbf{G}}_{n}^{l-1})^{-2}\}}{\pi_p^l(\mathbf{Q}_{p,n}^l(1))^2}
\frac{1}{N}\sum_{i\in\mathsf{S}_{p-1}^l}\{\mathbf{Q}_{p,n}^l(\mathbf{G}_n^l\pmb{\varphi}^l)(U_p^{l,i})- \mathbf{Q}_{p,n}^{l-1}(\mathbf{G}_n^{l-1}\pmb{\varphi}^{l-1})(\bar{U}_p^{l-1,i})\}^2\right]\\
T_7 & := &\overline{\mathbb{E}}\left[\frac{\max\{(\underline{\mathbf{G}}_{n}^l)^{-2},(\underline{\mathbf{G}}_{n}^{l-1})^{-2}\}}{\pi_p^l(\mathbf{Q}_{p,n}^l(1))^2}
\frac{1}{N}\sum_{i\in(\mathsf{S}_{p-1}^l)^c}\{\mathbf{Q}_{p,n}^l(\mathbf{G}_n^l\pmb{\varphi}^l)(U_p^{l,i})- \mathbf{Q}_{p,n}^{l-1}(\mathbf{G}_n^{l-1}\pmb{\varphi}^{l-1})(\bar{U}_p^{l-1,i})\}^2\right].
\end{eqnarray*}

For $T_6$, applying Cauchy-Schwarz (twice) with Lemma \ref{lem:res1} and conditional Jensen yields
$$
T_6 \leq C \overline{\mathbb{E}}\left[\frac{1}{N}\sum_{i\in\mathsf{S}_{p-1}^l}\{\mathbf{Q}_{p,n}^l(\mathbf{G}_n^l\pmb{\varphi}^l)(U_p^{l,i})- \mathbf{Q}_{p,n}^{l-1}(\mathbf{G}_n^{l-1}\pmb{\varphi}^{l-1})(\bar{U}_p^{l-1,i})\}^4\right]^{1/2}.
$$
Then conditioning upon the entire data trajectory and the information up-to resampling at time $p-1$, followed by Lemma \ref{lem:main_nlf_lem} gives the upper-bound
$$
T_6 \leq C(\|\varphi\|+\|\varphi\|_{\textrm{Lip}})^2\left(\overline{\mathbb{E}}\left[\frac{1}{N}\sum_{i\in \mathsf{S}_{p-1}^l} \|\widehat{X}_{p-1,p}^{l,i}-\widehat{\bar{X}}_{p-1,p}^{l-1,i}\|_2^4\right]+\Delta_l^{2}\right)^{1/2}.
$$
Applying Corollary \ref{cor:pairwise_difference} gives
\begin{equation}\label{eq:mistake2}
T_6 \leq C(\|\varphi\|+\|\varphi\|_{\textrm{Lip}})^2\Delta_l.
\end{equation}
For $T_7$, noting \eqref{eq:g_upper} and \eqref{eq:g_lower}  and taking expectations w.r.t.~$\{Y_t\}$ on the time interval $[p,n+1]$ one has the upper-bound
$$
T_7 \leq C(\|\varphi\|+\|\varphi\|_{\textrm{Lip}})^2\left(1-\overline{\mathbb{E}}\left[\frac{\textrm{Card}(\mathsf{S}_{p-1}^l)}{N}\right]\right).
$$
Then applying Lemma \ref{small_probability} one has
$$
T_7 \leq C(\|\varphi\|+\|\varphi\|_{\textrm{Lip}})^2\Delta_l^{1/2}.
$$
Thus, using \eqref{eq:mistake2}
$$
T_1 \leq C(\|\varphi\|+\|\varphi\|_{\textrm{Lip}})^2\Delta_l^{1/2}.
$$

For $T_2$ applying Cauchy-Schwarz and Corollary \ref{cor:nlf_cor1}
$$
T_2 \leq C\overline{\mathbb{E}}\left[\left(\frac{\mathbf{Q}_{p,n}^{l-1}(\mathbf{G}_n^{l-1}\pmb{\varphi}^{l-1})(\bar{U}_p^{l-1,1})
\max\{(\underline{\mathbf{G}}_{n}^l)^{-2},(\underline{\mathbf{G}}_{n}^{l-1})^{-2}\}
}{\pi_p^l(\mathbf{Q}_{p,n}^l(1))\pi_p^{l-1}(\mathbf{Q}_{p,n}^{l-1}(1))}\right)^4\right]^{1/2}\Delta_l.
$$
Applying H\"older (thrice) and Lemma \ref{lem:res1} one has
$$
T_2 \leq C(\|\varphi\|+\|\varphi\|_{\textrm{Lip}})^2\Delta_l.
$$
Hence we have shown that
$$
\overline{\mathbb{E}}[(\mathbf{D}_{p,n}^l(\mathbf{G}_n^l\pmb{\varphi}^l)(U_p^{l,1})- \mathbf{D}_{p,n}^{l-1}(\mathbf{G}_n^{l-1}\pmb{\varphi}^{l-1})(\bar{U}_p^{l-1,1}))^2] \leq C(\|\varphi\|+\|\varphi\|_{\textrm{Lip}})^2\Delta_l^{1/2}.
$$
\end{proof}


\begin{prop}\label{lem:l2_cpf_l2}
Assume (D\ref{hyp_diff:1}). Then for any $n\in\mathbb{N}_0$, there exists a $C<+\infty$ such that for any $(l,N,\varphi)\in\mathbb{N}\times\mathbb{N}\times\mathcal{B}_b(\mathbb{R}^{d_x})\cap\textrm{\emph{Lip}}_{\|\cdot\|_2}(\mathbb{R}^{d_x})$
$$
\overline{\mathbb{E}}\left[\left([\eta_{n}^{l}-\eta_{n}^{l-1}]^{N}(\varphi)-[\eta_{n}^{l}-\eta_{n}^{l-1}](\varphi)\right)^2\right] \leq C(\|\varphi\|+\|\varphi\|_{\textrm{\emph{Lip}}})^2\left(\frac{\Delta_l^{1/2}}{N}+\frac{\Delta_l^{1/4}}{N^{3/2}}\right).
$$
\end{prop}

\begin{proof}
The result follows, essentially, by using \cite[Lemma C.5]{mlpf} along with Lemmata \ref{lem:l2_cpf}, \ref{lem:l2_cpf_contr}, along with Corollary \ref{cor:nlf_cor1} (see also Remark \ref{rem:extra_nlf_res}).
\end{proof}

%

\subsection{Results for the Non-Linear Filter}\label{app:nl_filt}

In this section, we consider the case of the non-linear filter, with a probability space $(\Omega,\mathcal{F})$, with $\mathcal{F}_t$ the filtration, that includes $\{Y_t\}_{t\geq 0}$ as standard Brownian motion independent of a diffusion process $\{X_t^{x}\}_{t\geq 0}$ which obeys \eqref{eq:state}
with initial condition $x\in\mathbb{R}^{d_x}$ and associated Euler discretization (with the same Brownian increments) at level $l$ $(\widetilde{X}_{\Delta_l}^x,\widetilde{X}_{2\Delta_l}^{x},\dots)$. We will also consider another diffusion process
$\{X_t^{x_{*}}\}_{t\geq 0}$ which obeys \eqref{eq:state}, initial condition ${x_{*}}\in\mathbb{R}^{d_x}$ and the same Brownian motion as $\{X_t^x\}_{t\geq 0}$ and associated Euler discretization (with the same Brownian increments) at level $l$ $(\widetilde{X}_{\Delta_l}^{x_{*}},\widetilde{X}_{2\Delta_l}^{x_{*}},\dots)$. 
Expectations are written $\overline{\mathbb{E}}$. 
We set for $(p,n)\in\mathbb{N}_0^2$, $n+1>p$
$$
Z_{p,n+1}^x = \exp\Big\{\int_{p}^{n+1}h(X_s^x)^*dY_s -\frac{1}{2}\int_{p}^{n+1}h(X_s^x)^*h(X_s^x)ds\Big\}
$$
with the convention that $Z_{0,n+1}^x=Z_{n+1}^x$.
The technical results in this appendix are critical in proving the results in Appendix \ref{app:cpf}. Although some of the results are more-or-less known in the literature (e.g.~\cite{picard}), we give the proofs for the completeness of the article.


\begin{lem}\label{lem:nlf_lem1}
Assume (D\ref{hyp_diff:1}). Then for any $(n,r)\in\mathbb{N}_0\times\mathbb{N}$, there exists a $C<+\infty$ such that for any $(l,\varphi,x)\in\mathbb{N}_0\times\mathcal{B}_b(\mathbb{R}^{d_x})\cap\textrm{\emph{Lip}}_{\|\cdot\|_2}(\mathbb{R}^{d_x})\times\mathbb{R}^{d_x}$
$$
\overline{\mathbb{E}}\Big[\Big|\varphi(\widetilde{X}_{n+1}^x)Z_{n+1}^l(\widetilde{X}_0^x,\widetilde{X}_{\Delta_l}^x,\dots,\widetilde{X}_{n+1-\Delta_l}^x)-
\varphi(X_{n+1}^x)Z_{n+1}^x\Big|^r\Big]^{1/r} \leq C(\|\varphi\|+\|\varphi\|_{\textrm{\emph{Lip}}})\Delta_l^{1/2}.
$$
\end{lem}

\begin{proof}
We have
\begin{equation}\label{eq:nlf_main}
\overline{\mathbb{E}}\Big[\Big|\varphi(\widetilde{X}_{n+1}^x)Z_{n+1}^l(\widetilde{X}_0^x,\widetilde{X}_{\Delta_l}^x,\dots,\widetilde{X}_{n+1-\Delta_l}^x)-
\varphi(X_{n+1}^x)Z_{n+1}^x\Big|^r\Big] \leq C(T_1+T_2)
\end{equation}
where
\begin{eqnarray*}
T_1 & := & \overline{\mathbb{E}}\Big[\Big|\varphi(\widetilde{X}_{n+1}^x)Z_{n+1}^l(\widetilde{X}_0^x,\widetilde{X}_{\Delta_l}^x,\dots,\widetilde{X}_{n+1-\Delta_l}^x)-
\varphi(X_{n+1}^x)Z_{n+1}^l(X_0^x,X_{\Delta_l}^x,\dots,X_{n+1-\Delta_l}^x)\Big|^r\Big] \\
T_2 & := & \overline{\mathbb{E}}\Big[\Big|\varphi(X_{n+1}^x)Z_{n+1}^l(X_0^x,X_{\Delta_l}^x,\dots,X_{n+1-\Delta_l}^x)-
\varphi(X_{n+1}^x)Z_{n+1}^x\Big|^r\Big].
\end{eqnarray*}
The term $T_2$ can be treated with a very similar proof to $T_1$ along the lines of \cite[Theorem 21.3]{crisan}, so we will give a proof for $T_1$ only.

One has
\begin{equation}\label{eq:nlf_main_1}
T_1 \leq C(T_3 + T_4)
\end{equation}
where
\begin{eqnarray*}
T_3 & := & \overline{\mathbb{E}}\Big[\Big|[\varphi(\widetilde{X}_{n+1}^x)-\varphi(X_{n+1}^x)]Z_{n+1}^l(\widetilde{X}_0^x,\widetilde{X}_{\Delta_l}^x,\dots,\widetilde{X}_{n+1-\Delta_l}^x)\Big|^r\Big] \\
T_4 & := & \overline{\mathbb{E}}\Big[\Big|\varphi(X_{n+1}^x)\Big(Z_{n+1}^l(\widetilde{X}_0^x,\widetilde{X}_{\Delta_l}^x,\dots,\widetilde{X}_{n+1-\Delta_l}^x)-Z_{n+1}^l(X_0^x,X_{\Delta_l}^x,\dots,X_{n+1-\Delta_l}^x)\Big)\Big|^r\Big].
\end{eqnarray*}
We now need to appropriately upper-bound $T_3$ and $T_4$.
For $T_3$, taking expectations of $Z_{n+1}^l(\widetilde{X}_0^x,\widetilde{X}_{\Delta_l}^x,\dots,\widetilde{X}_{n+1-\Delta_l}^x)$ w.r.t.~the process $\{Y_t\}$ and using the fact that $h$ is bounded along with the
fact that $\varphi\in\textrm{Lip}_{\|\cdot\|_2}(\mathbb{R}^{d_x})$ gives the upper-bound
$$
T_3 \leq C(\|\varphi\|+\|\varphi\|_{\textrm{Lip}})^r \overline{\mathbb{E}}[\|\widetilde{X}_{n+1}^x-X_{n+1}^x\|_2^r].
$$
Then using standard results on Euler discretization of diffusion processes (e.g.~\cite{kloeden})
\begin{equation}\label{eq:nlf_t3}
T_3 \leq C(\|\varphi\|+\|\varphi\|_{\textrm{Lip}})^r\Delta_l^{r/2}.
\end{equation}

For $T_4$ as $\varphi\in\mathcal{B}_b(\mathbb{R}^{d_x})$, one has
\begin{equation}\label{eq:nlf_t4}
T_4 \leq (\|\varphi\|+\|\varphi\|_{\textrm{Lip}})^r\overline{\mathbb{E}}\Big[\Big|Z_{n+1}^l(\widetilde{X}_0^x,\widetilde{X}_{\Delta_l}^x,\dots,\widetilde{X}_{n+1-\Delta_l}^x)-Z_{n+1}^l(X_0^x,X_{\Delta_l}^x,\dots,X_{n+1-\Delta_l}^x)\Big|^r\Big].
\end{equation}
Now, by the Mean Value Theorem (MVT)
$$
Z_{n+1}^l(\widetilde{X}_0^x,\widetilde{X}_{\Delta_l}^x,\dots,\widetilde{X}_{n+1-\Delta_l}^x)-Z_{n+1}^l(X_0^x,X_{\Delta_l}^x,\dots,X_{n+1-\Delta_l}^x)
 = 
\Big(
H_{n+1}^l(\widetilde{X}_0^x,\widetilde{X}_{\Delta_l}^x,\dots,\widetilde{X}_{n+1-\Delta_l}^x)-
$$
\begin{equation}
H_{n+1}^l(X_0^x,X_{\Delta_l}^x,\dots,X_{n+1-\Delta_l}^x)
\Big)
\int_{0}^1 \widetilde{H}_{n+1}^{l,s}(\widetilde{X}_0^x,\widetilde{X}_{\Delta_l}^x,\dots,\widetilde{X}_{n+1-\Delta_l}^x,X_0^x,X_{\Delta_l}^x,\dots,X_{n+1-\Delta_l}^x) ds
\label{eq:z_diff_mvt}
\end{equation}
where for any $(l,n,(x_0,x_{\Delta_l},\dots,x_{n+1-\Delta_l}),(x_0',x_{\Delta_l}',\dots,x_{n+1-\Delta_l}'),s)\in\{0,1,\dots\}^2\times(\mathbb{R}^{d_x})^{2(n+1)\Delta_l^{-1}}\times[0,1]$
\begin{eqnarray*}
H_{n+1}^l(x_0,x_{\Delta_l},\dots,x_{n+1-\Delta_l}) & = & \log[Z_{n+1}^l(x_0,x_{\Delta_l},\dots,x_{n+1-\Delta_l})] \\
\widetilde{H}_{n+1}^{l,s}(x_0,x_{\Delta_l},\dots,x_{n+1-\Delta_l},x_0',x_{\Delta_l}',\dots,x_{n+1-\Delta_l}')) & = & \exp\{sH_{n+1}^l(x_0,x_{\Delta_l},\dots,x_{n+1-\Delta_l})+\\ & & (1-s)H_{n+1}^l(x_0',x_{\Delta_l}',\dots,x_{n+1-\Delta_l}')\}.
\end{eqnarray*}
Then, by using \eqref{eq:z_diff_mvt} in \eqref{eq:nlf_t4} and applying Cauchy-Schwarz
$$
T_4 \leq (\|\varphi\|+\|\varphi\|_{\textrm{Lip}})^r \overline{\mathbb{E}}\Big[\Big|\int_{0}^1 \widetilde{H}_{n+1}^{l,s}(\widetilde{X}_0^x,\widetilde{X}_{\Delta_l}^x,\dots,\widetilde{X}_{n+1-\Delta_l}^x,X_0^x,X_{\Delta_l}^x,\dots,X_{n+1-\Delta_l}^x) ds\Big|^{2r}\Big]^{1/2}\times
$$
\begin{equation}
\overline{\mathbb{E}}\Big[\Big|H_{n+1}^l(\widetilde{X}_0^x,\widetilde{X}_{\Delta_l}^x,\dots,\widetilde{X}_{n+1-\Delta_l}^x)-H_{n+1}^l(X_0^x,X_{\Delta_l}^x,\dots,X_{n+1-\Delta_l}^x)\Big|^{2r}\Big]^{1/2}.
\label{eq:nlf_t4_1}
\end{equation}
Now, taking expectations w.r.t.~the process $\{Y_t\}$ and using the fact that $h$ is bounded, there exists a $C<+\infty$ such that
$$
\sup_{l\geq 0}\sup_{s\in[0,1]}
\overline{\mathbb{E}}\Big[\Big|\widetilde{H}_{n+1}^{l,s}(\widetilde{X}_0^x,\widetilde{X}_{\Delta_l}^x,\dots,\widetilde{X}_{n+1-\Delta_l}^x,X_0^x,X_{\Delta_l}^x,\dots,X_{n+1-\Delta_l}^x)\Big|^{2r}\Big] \leq C
$$
so, via Jensen, we need only deal with the right-most expectation on the R.H.S~of \eqref{eq:nlf_t4_1}, call it $T_5$. Now
$$
H_{n+1}^l(\widetilde{X}_0^x,\widetilde{X}_{\Delta_l}^x,\dots,\widetilde{X}_{n+1-\Delta_l}^x)-H_{n+1}^l(X_0^x,X_{\Delta_l}^x,\dots,X_{n+1-\Delta_l}^x)= M^l_{n+1} - R^l_{n+1}
$$
where
\begin{eqnarray*}
M^l_{n+1}& := & \sum_{k=0}^{\Delta_l^{-1}(n+1)-1}\Big[\{h(\widetilde{X}_{k\Delta_l})^*-h(X_{k\Delta_l})^*\}(Y_{(k+1)\Delta_l}-Y_{k\Delta_l})\Big] \\
R^l_{n+1} &:= & \frac{\Delta_l}{2}\sum_{k=0}^{\Delta_l^{-1}(n+1)-1}\Big[h(\widetilde{X}_{k\Delta_l})^*h(\widetilde{X}_{k\Delta_l})-h(X_{k\Delta_l})^*h(X_{k\Delta_l}))\Big]
\end{eqnarray*}
and we set $M_0=R_0=0$.
Thus, applying the $C_{2r}-$inequality, one has 
\begin{equation}\label{eq:pref_ineq1}
T_5^2 \leq C\Big(\overline{\mathbb{E}}[|M^l_{n+1}|^{2r}] + \overline{\mathbb{E}}[|R^l_{n+1}|^{2r}]\Big).
\end{equation}
We first focus on the first term on the R.H.S.~of \eqref{eq:pref_ineq1}. Applying $C_{2r}-$inequality $d_y-$times, we have
\begin{equation}\label{eq:pref_ineq2}
\overline{\mathbb{E}}[|M^l_{n+1}|^{2r}] \leq C\sum_{i=1}^{d_y}\overline{\mathbb{E}}\Big[\Big|
\sum_{k=0}^{\Delta_l^{-1}(n+1)-1}\Big[\{h^{(i)}(\widetilde{X}_{k\Delta_l})-h^{(i)}(X_{k\Delta_l})\}(Y_{(k+1)\Delta_l}^{(i)}-Y_{k\Delta_l}^{(i)})\Big]\Big|^{2r}\Big].
\end{equation}
We consider just the $i^{\textrm{th}}$ summand on the R.H.S., as the argument to be used is essentially exchangeable w.r.t.~$i$. As $\{M^l_{n},\mathcal{F}_{n\Delta_{l}^{-1}}\}_{n\in\{0,1,\dots\}}$ is a Martingale, applying the Burkholder-Gundy-Davis (BGD) inequality, Minkowski inequality, along with $h^{(i)}\in\textrm{Lip}_{\|\cdot\|_2}(\mathbb{R}^{d_x})$:
$$
\overline{\mathbb{E}}\Big[\Big|
\sum_{k=0}^{\Delta_l^{-1}(n+1)-1}\Big[\{h^{(i)}(\widetilde{X}_{k\Delta_l})-h^{(i)}(X_{k\Delta_l})\}(Y_{(k+1)\Delta_l}^{(i)}-Y_{k\Delta_l}^{(i)})\Big]\Big|^{2r}\Big]
\leq C\Delta_l^{r}\Big(
\sum_{k=0}^{\Delta_l^{-1}(n+1)-1}\overline{\mathbb{E}}[\|\widetilde{X}_{k\Delta_l}-X_{k\Delta_l}\|_2^{2r}]^{1/r}
\Big)^{r}.
$$
Then using standard results on Euler discretization of diffusion processes:
$$
\overline{\mathbb{E}}\Big[\Big|
\sum_{k=0}^{\Delta_l^{-1}(n+1)-1}\Big[\{h^{(i)}(\widetilde{X}_{k\Delta_l})-h^{(i)}(X_{k\Delta_l})\}(Y_{(k+1)\Delta_l}^{(i)}-Y_{k\Delta_l}^{(i)})\Big]\Big|^{2r}\Big]
\leq C\Delta_l^{r}.
$$
Thus, on returning to \eqref{eq:pref_ineq2}, we have shown that
\begin{equation}\label{eq:pref_ineq3}
\overline{\mathbb{E}}[|M^l_{n+1}|^{2r}] \leq C\Delta_l^{r}.
\end{equation}
Noting that as $h^{(i)}\in\textrm{Lip}_{\|\cdot\|_2}(\mathbb{R}^{d_x})$ and $h^{(i)}\in\mathcal{B}_b(\mathbb{R}^{d_x})$, it follows that $(h^{(i)})^2\in\textrm{Lip}_{\|\cdot\|_2}(\mathbb{R}^{d_x})$. So using very similar
calculations to those for $M^l_{n+1}$ (except not requiring to apply the BGD inequality), one can prove that
\begin{equation}\label{eq:pref_ineq4}
\overline{\mathbb{E}}[|R^l_{n+1}|^{2r}] \leq C\Delta_l^{r}.
\end{equation}
Thus combining \eqref{eq:pref_ineq3}-\eqref{eq:pref_ineq4} with \eqref{eq:pref_ineq1}, one has that
$T_5 \leq C\Delta_l^{r/2}$ and hence that
\begin{equation}\label{eq:nlf_t4_final}
T_4 \leq C(\|\varphi\|+\|\varphi\|_{\textrm{Lip}})^r\Delta_l^{r/2}.
\end{equation}
Noting \eqref{eq:nlf_main_1} and using the bounds \eqref{eq:nlf_t3} and \eqref{eq:nlf_t4_final}
$$
T_1 \leq C(\|\varphi\|+\|\varphi\|_{\textrm{Lip}})^r\Delta_l^{r/2}.
$$
As noted above, a similar bound can be obtained for $T_2$ and noting \eqref{eq:nlf_main}, the proof is hence concluded.
\end{proof}

\begin{lem}\label{lem:nlf_lem2}
Assume (D\ref{hyp_diff:1}). Then for any $(p,n,r)\in\mathbb{N}_0^2\times\mathbb{N}$, $n>p$ there exists a $C<+\infty$ such that for any $(x,x_*)\in\mathbb{R}^{d_x}\times\mathbb{R}^{d_x}$
$$
\overline{\mathbb{E}}\Big[\Big|Z_{p,n}^x-Z_{p,n}^{x_*}\Big|^r\Big]^{1/r} \leq C\|x-x_{*}\|_2.
$$
\end{lem}

\begin{proof}
This result can be proved in a similar manner to considering \eqref{eq:z_diff_mvt} in the proof of Lemma \ref{lem:nlf_lem1}, that is by using the MVT and a Martingale plus remainder method.
The main difference is that one must use the result (which can be deduced by \cite[Corollary v.11.7]{rogers} and the Gr\"onwall's inequality)
\begin{equation}\label{eq:cont_exp_diff}
\sup_{t\in[p,n]}\overline{\mathbb{E}}[\|X_t^x-X_{t}^{x_*}\|_2^{2r}]^{1/(2r)} \leq C\|x-x_{*}\|_2.
\end{equation}
The proof is omitted due to the similarity to the proof associated to \eqref{eq:z_diff_mvt}.
\end{proof}

\begin{lem}\label{lem:nlf_lem3}
Assume (D\ref{hyp_diff:1}). Then for any $(n,r)\in\mathbb{N}_0\times\mathbb{N}$, there exists a $C<+\infty$ such that for any $(\varphi,x,x_{*})\in\mathcal{B}_b(\mathbb{R}^{d_x})\cap\textrm{\emph{Lip}}_{\|\cdot\|_2}(\mathbb{R}^{d_x})\times\mathbb{R}^{d_x}\times\mathbb{R}^{d_x}$
$$
\overline{\mathbb{E}}\Big[\Big|\varphi(X_{n+1}^x)Z_{n+1}^x-
\varphi(X_{n+1}^{x_*})Z_{n+1}^{x_*}\Big|^r\Big]^{1/r} \leq C(\|\varphi\|+\|\varphi\|_{\textrm{\emph{Lip}}})\|x-x_*\|_2.
$$
\end{lem}

\begin{proof}
We have
$$
\overline{\mathbb{E}}\Big[\Big|\varphi(X_{n+1}^x)Z_{n+1}^x-
\varphi(X_{n+1}^{x_*})Z_{n+1}^{x_*}\Big|^r\Big]^{1/r} \leq T_1 + T_2
$$
where
\begin{eqnarray*}
T_1 & := & \overline{\mathbb{E}}\Big[\Big|\{\varphi(X_{n+1}^x)-\varphi(X_{n+1}^{x_*})\}Z_{n+1}^x\Big|^r\Big]^{1/r}\\
T_2 & := & \overline{\mathbb{E}}\Big[\Big|\varphi(X_{n+1}^{x_*})\{Z_{n+1}^x-Z_{n+1}^{x_*}\}\Big|^r\Big]^{1/r}.
\end{eqnarray*}
So we proceed to control the two terms in $T_1$ and $T_2$.

For $T_1$, apply Cauchy-Schwarz to obtain the upper-bound
$$
T_1 \leq \overline{\mathbb{E}}[|Z_{n+1}^{x}|^{2r}]^{1/(2r)}\mathbb{\overline{E}}\Big[\Big|\{\varphi(X_{n+1}^x)-\varphi(X_{n+1}^{x_*})\}\Big|^{2r}\Big]^{1/(2r)}.
$$
As $\overline{\mathbb{E}}[|Z_{n+1}^{x}|^{2r}]^{1/(2r)}\leq C$
and using $\varphi\in\textrm{Lip}_{\|\cdot\|_2}(\mathbb{R}^{d_x})$ along with \eqref{eq:cont_exp_diff} yields
\begin{equation}\label{eq:nlf_lem3_2}
T_1 \leq C(\|\varphi\|+\|\varphi\|_{\textrm{Lip}})\|x-x_{*}\|_2.
\end{equation}
For $T_2$ using $\varphi\in\mathcal{B}_b(\mathbb{R}^{d_x})$
$$
T_2 \leq (\|\varphi\|+\|\varphi\|_{\textrm{Lip}})\overline{\mathbb{E}}\Big[\Big|Z_{n+1}^x-Z_{n+1}^{x_*}\Big|^{r}\Big]^{1/r}.
$$
Applying Lemma \ref{lem:nlf_lem2} gives
\begin{equation}\label{eq:nlf_lem3_3}
T_2 \leq C(\|\varphi\|+\|\varphi\|_{\textrm{Lip}})\|x-x_{*}\|_2.
\end{equation}
Noting \eqref{eq:nlf_lem3_2} and \eqref{eq:nlf_lem3_3} allows one to conclude.
\end{proof}

\begin{lem}\label{lem:main_nlf_lem}
Assume (D\ref{hyp_diff:1}). Then for any $(n,r)\in\mathbb{N}_0\times\mathbb{N}$, there exists a $C<+\infty$ such that for any $(l,\varphi,x,x_{*})\in\mathbb{N}\times\mathcal{B}_b(\mathbb{R}^{d_x})\cap\textrm{\emph{Lip}}_{\|\cdot\|_2}(\mathbb{R}^{d_x})\times\mathbb{R}^{d_x}\times\mathbb{R}^{d_x}$
$$
\overline{\mathbb{E}}\Big[\Big|\varphi(\widetilde{X}_{n+1}^x)Z_{n+1}^l(\widetilde{X}_0^x,\widetilde{X}_{\Delta_l}^x,\dots,\widetilde{X}_{n+1-\Delta_l}^x)-
\varphi(\widetilde{X}_{n+1}^{x_*})Z_{n+1}^{l-1}(\widetilde{X}_0^{x_*},\widetilde{X}_{\Delta_{l-1}}^{x_*},\dots,\widetilde{X}_{n+1-\Delta_{l-1}}^{x_*})
\Big|^r\Big]^{1/r}
 \leq 
$$
$$
C(\|\varphi\|+\|\varphi\|_{\textrm{\emph{Lip}}})\Big(\Delta_l^{1/2}+\|x-x_*\|_2\Big).
$$
\end{lem}

\begin{proof}
The expectation in the statement of the Lemma is upper-bounded by $\sum_{j=1}^3 T_j$ where
\begin{eqnarray*}
T_1 & := &  \overline{\mathbb{E}}\Big[\Big|\varphi(\widetilde{X}_{n+1}^x)Z_{n+1}^l(\widetilde{X}_0^x,\widetilde{X}_{\Delta_l}^x,\dots,\widetilde{X}_{n+1-\Delta_l}^x)-
\varphi(X_{n+1}^x)Z_{n+1}^x\Big|^r\Big]^{1/r} \\
T_2 & := & \overline{\mathbb{E}}\Big[\Big|
\varphi(X_{n+1}^x)Z_{n+1}^x-
\varphi(X_{n+1}^{x_*})Z_{n+1}^{x_*}
\Big|^r\Big]^{1/r} \\
T_3 & := & \overline{\mathbb{E}}\Big[\Big|\varphi(\widetilde{X}_{n+1}^{x_*})Z_{n+1}^{l-1}(\widetilde{X}_0^{x_*},\widetilde{X}_{\Delta_{l-1}}^{x_*},\dots,\widetilde{X}_{n+1-\Delta_{l-1}}^{x_*})-
\varphi(X_{n+1}^{x_*})Z_{n+1}^{x_*}\Big|^r\Big]^{1/r}.
\end{eqnarray*}
The proof is completed by applying Lemma \ref{lem:nlf_lem1} to $T_1$ and $T_3$, and Lemma \ref{lem:nlf_lem3} to $T_2$.
\end{proof}

\begin{lem}\label{lem:nlf_lem5}
Assume (D\ref{hyp_diff:1}). Then for any $(n,p,r)\in\mathbb{N}_0\times\mathbb{N}^2$, there exists a $C<+\infty$ such that for any $(l,\varphi,x)\in\mathbb{N}_0\times\mathcal{B}_b(\mathbb{R}^{d_x})\cap\textrm{\emph{Lip}}_{\|\cdot\|_2}(\mathbb{R}^{d_x})\times\mathbb{R}^{d_x}$
$$
\overline{\mathbb{E}}\Bigg[\Bigg|\frac{\overline{\mathbb{E}}[\varphi(\widetilde{X}_{n+1}^x)Z_{n}^l(\widetilde{X}_0^x,\widetilde{X}_{\Delta_l}^x,\dots,\widetilde{X}_{n-\Delta_l}^x)|\mathcal{Y}_n]}
{\overline{\mathbb{E}}[Z_{p}^l(\widetilde{X}_0^x,\widetilde{X}_{\Delta_l}^x,\dots,\widetilde{X}_{p-\Delta_l}^x)|\mathcal{Y}_{p}]}-\frac{\overline{\mathbb{E}}[\varphi(X_{n+1}^x)Z_{n}^x|\mathcal{Y}_{n}]}{\overline{\mathbb{E}}[Z_{p}^x|\mathcal{Y}_{p}]}\Bigg|^r\Bigg]^{1/r}
\leq C(\|\varphi\|+\|\varphi\|_{\textrm{\emph{Lip}}})\Delta_l^{1/2}.
$$
\end{lem}

\begin{proof}
We have
$$
\overline{\mathbb{E}}\Bigg[\Bigg|\frac{\overline{\mathbb{E}}[\varphi(\widetilde{X}_{n+1}^x)Z_{n}^l(\widetilde{X}_0^x,\widetilde{X}_{\Delta_l}^x,\dots,\widetilde{X}_{n-\Delta_l}^x)|\mathcal{Y}_n]}
{\overline{\mathbb{E}}[Z_{p}^l(\widetilde{X}_0^x,\widetilde{X}_{\Delta_l}^x,\dots,\widetilde{X}_{p-\Delta_l}^x)|\mathcal{Y}_{p}]}-\frac{\overline{\mathbb{E}}[\varphi(X_{n+1}^x)Z_{n}^x|\mathcal{Y}_{n}]}{\overline{\mathbb{E}}[Z_{p}^x|\mathcal{Y}_{p}]}\Bigg|^r\Bigg]^{1/r} \leq T_1 + T_2
$$
where
\begin{eqnarray*}
T_1 & := & \overline{\mathbb{E}}\Bigg[\Bigg|\frac{\overline{\mathbb{E}}[\varphi(\widetilde{X}_{n+1}^x)Z_{n}^l(\widetilde{X}_0^x,\widetilde{X}_{\Delta_l}^x,\dots,\widetilde{X}_{n-\Delta_l}^x)|\mathcal{Y}_n]}
{\overline{\mathbb{E}}[Z_{p}^l(\widetilde{X}_0^x,\widetilde{X}_{\Delta_l}^x,\dots,\widetilde{X}_{p-\Delta_l}^x)|\mathcal{Y}_{p}]} - 
\frac{\overline{\mathbb{E}}[\varphi(X_{n+1}^x)Z_{n}^l(X_0^x,X_{\Delta_l}^x,\dots,X_{n-\Delta_l}^x)|\mathcal{Y}_n]}
{\overline{\mathbb{E}}[Z_{p}^l(X_0^x,X_{\Delta_l}^x,\dots,X_{p-\Delta_l}^x)|\mathcal{Y}_{p}]}\Bigg|^r\Bigg]^{1/r} \\
T_2 & = & \overline{\mathbb{E}}\Bigg[\Bigg|\frac{\overline{\mathbb{E}}[\varphi(X_{n+1}^x)Z_{n}^l(X_0^x,X_{\Delta_l}^x,\dots,X_{n-\Delta_l}^x)|\mathcal{Y}_n]}
{\overline{\mathbb{E}}[Z_{p}^l(X_0^x,X_{\Delta_l}^x,\dots,X_{p-\Delta_l}^x)|\mathcal{Y}_{p}]} - 
\frac{\overline{\mathbb{E}}[\varphi(X_{n+1}^x)Z_{n}^x|\mathcal{Y}_{n}]}{\overline{\mathbb{E}}[Z_{p}^x|\mathcal{Y}_{p}]}\Bigg|^r\Bigg]^{1/r}.
\end{eqnarray*}
$T_2$ can be dealt with in a similar way to $T_1$, except one uses approaches similar to \cite[Theorem 21.3]{crisan} (which is a similar MVT, Martingale plus remainder method that has been used in the proof of Lemma \ref{lem:nlf_lem1}), so we treat the former only.

Now, we have
$$
T_1 \leq T_3 + T_4
$$
where\begin{eqnarray*}
T_3 & := & \overline{\mathbb{E}}\Bigg[\Bigg|\frac{\overline{\mathbb{E}}[\varphi(\widetilde{X}_{n+1}^x)Z_{n}^l(\widetilde{X}_0^x,\widetilde{X}_{\Delta_l}^x,\dots,\widetilde{X}_{n-\Delta_l}^x)|\mathcal{Y}_n]}{\overline{\mathbb{E}}[Z_{p}^l(\widetilde{X}_0^x,\widetilde{X}_{\Delta_l}^x,\dots,\widetilde{X}_{p-\Delta_l}^x)|\mathcal{Y}_{p}]
\overline{\mathbb{E}}[Z_{p}^l(X_0^x,X_{\Delta_l}^x,\dots,X_{p-\Delta_l}^x)|\mathcal{Y}_{p}]}\Big(\overline{\mathbb{E}}[Z_{p}^l(X_0^x,X_{\Delta_l}^x,\dots,X_{p-\Delta_l}^x)|\mathcal{Y}_{p}]-\\ & & \overline{\mathbb{E}}[Z_{p}^l(\widetilde{X}_0^x,\widetilde{X}_{\Delta_l}^x,\dots,\widetilde{X}_{p-\Delta_l}^x)|\mathcal{Y}_{p}]\Big)
\Bigg|^r\Bigg]^{1/r}\\
T_4 & := & \overline{\mathbb{E}}\Bigg[\Bigg|\frac{1}{\overline{\mathbb{E}}[Z_{p}^l(X_0^x,X_{\Delta_l}^x,\dots,X_{p-\Delta_l}^x)|\mathcal{Y}_{p}]}\Big(\overline{\mathbb{E}}[\varphi(\widetilde{X}_{n+1}^x)Z_{n}^l(\widetilde{X}_0^x,\widetilde{X}_{\Delta_l}^x,\dots,\widetilde{X}_{n-\Delta_l}^x)|\mathcal{Y}_n]- \\ & &
\overline{\mathbb{E}}[\varphi(X_{n+1}^x)Z_{n}^l(X_0^x,X_{\Delta_l}^x,\dots,X_{n-\Delta_l}^x)|\mathcal{Y}_n]\Big)
\Bigg|^r\Bigg]^{1/r}.
\end{eqnarray*}
For $T_3$ applying Cauchy-Schwarz and conditional Jensen,
$$
T_3 \leq \overline{\mathbb{E}}\Bigg[\Bigg|\frac{\overline{\mathbb{E}}[\varphi(\widetilde{X}_{n+1}^x)Z_{n}^l(\widetilde{X}_0^x,\widetilde{X}_{\Delta_l}^x,\dots,\widetilde{X}_{n-\Delta_l}^x)|\mathcal{Y}_n]}{\overline{\mathbb{E}}[Z_{p}^l(\widetilde{X}_0^x,\widetilde{X}_{\Delta_l}^x,\dots,\widetilde{X}_{p-\Delta_l}^x)|\mathcal{Y}_{p}]
\overline{\mathbb{E}}[Z_{p}^l(X_0^x,X_{\Delta_l}^x,\dots,X_{p-\Delta_l}^x)|\mathcal{Y}_{p}]}\Bigg|^{2r}\Bigg]^{1/(2r)}\times
$$
$$
\overline{\mathbb{E}}\Big[\Big|Z_{p}^l(X_0^x,X_{\Delta_l}^x,\dots,X_{p-\Delta_l}^x)-Z_{p}^l(\widetilde{X}_0^x,\widetilde{X}_{\Delta_l}^x,\dots,\widetilde{X}_{p-\Delta_l}^x)\Big|^{2r}\Big]^{1/(2r)}.
$$
For the left-most expectation on the R.H.S.~one can use $\varphi\in\mathcal{B}_b(\mathbb{R}^{d_x})$, the H\"older and Jensen inequalities along with the proof approaches in the proof of Lemma \ref{lem:res1} to establish that the expectation is upper-bounded by a constant $C$ that does not depend upon $l,x$.
For the right-most expectation on the R.H.S.~one can use the ideas in \eqref{eq:z_diff_mvt} to deduce that
$$
T_3 \leq  C(\|\varphi\|+\|\varphi\|_{\textrm{Lip}})\Delta_l^{1/2}.
$$
The proof for $T_4$ is similar, except using the ideas for \eqref{eq:nlf_main_1} instead of \eqref{eq:z_diff_mvt}. That is, $T_4 \leq  C(\|\varphi\|+\|\varphi\|_{\textrm{Lip}})\Delta_l^{1/2}$. Hence
$$
T_1 \leq  C(\|\varphi\|+\|\varphi\|_{\textrm{Lip}})\Delta_l^{1/2}.
$$
This completes the argument.
\end{proof}

\begin{rem}\label{rem:theo_bias}
One can easily deduce: Assume (D\ref{hyp_diff:1}). Then for any $(n,r)\in\mathbb{N}_0\times\mathbb{N}$, there exists a $C<+\infty$ such that for any $(l,\varphi,x)\in\mathbb{N}_0\times\mathcal{B}_b(\mathbb{R}^{d_x})\cap\textrm{\emph{Lip}}_{\|\cdot\|_2}(\mathbb{R}^{d_x})\times\mathbb{R}^{d_x}$
$$
\overline{\mathbb{E}}\Bigg[\Bigg|\frac{\overline{\mathbb{E}}[\varphi(\widetilde{X}_{n}^x)Z_{n}^l(\widetilde{X}_0^x,\widetilde{X}_{\Delta_l}^x,\dots,\widetilde{X}_{n-\Delta_l}^x)|\mathcal{Y}_n]}
{\overline{\mathbb{E}}[Z_{n}^l(\widetilde{X}_0^x,\widetilde{X}_{\Delta_l}^x,\dots,\widetilde{X}_{n-\Delta_l}^x)|\mathcal{Y}_{n}]}-\frac{\overline{\mathbb{E}}[\varphi(X_{n}^x)Z_{n}^x|\mathcal{Y}_{n}]}{\overline{\mathbb{E}}[Z_{n}^x|\mathcal{Y}_{n}]}\Bigg|^r\Bigg]^{1/r}
\leq C(\|\varphi\|+\|\varphi\|_{\textrm{\emph{Lip}}})\Delta_l^{1/2}.
$$
\end{rem}

\begin{cor}\label{cor:nlf_cor1}
Assume (D\ref{hyp_diff:1}). Then for any $(n,p,r)\in\mathbb{N}_0\times\mathbb{N}^2$, there exists a $C<+\infty$ such that for any $(l,\varphi,x)\in\mathbb{N}\times\mathcal{B}_b(\mathbb{R}^{d_x})\cap\textrm{\emph{Lip}}_{\|\cdot\|_2}(\mathbb{R}^{d_x})\times\mathbb{R}^{d_x}$
$$
\overline{\mathbb{E}}\Bigg[\Bigg|\frac{\overline{\mathbb{E}}[\varphi(\widetilde{X}_{n+1}^x)Z_{n}^l(\widetilde{X}_0^x,\widetilde{X}_{\Delta_l}^x,\dots,\widetilde{X}_{n-\Delta_l}^x)|\mathcal{Y}_n]}
{\overline{\mathbb{E}}[Z_{p}^l(\widetilde{X}_0^x,\widetilde{X}_{\Delta_l}^x,\dots,\widetilde{X}_{p-\Delta_l}^x)|\mathcal{Y}_{p}]}-
\frac{\overline{\mathbb{E}}[\varphi(\widetilde{X}_{n+1}^x)Z_{n}^{l-1}(\widetilde{X}_0^x,\widetilde{X}_{\Delta_{l-1}}^x,\dots,\widetilde{X}_{n-\Delta_{l-1}}^x)|\mathcal{Y}_n]}
{\overline{\mathbb{E}}[Z_{p}^{l-1}(\widetilde{X}_0^x,\widetilde{X}_{\Delta_{l-1}}^x,\dots,\widetilde{X}_{p-\Delta_{l-1}}^x)|\mathcal{Y}_{p}]}
\Bigg|^r\Bigg]^{1/r}
$$
$$
\leq C(\|\varphi\|+\|\varphi\|_{\textrm{\emph{Lip}}})\Delta_l^{1/2}.
$$
\end{cor}

\begin{proof}
Can be easily proved by using Lemma \ref{lem:nlf_lem5}.
\end{proof}

\begin{rem}\label{rem:extra_nlf_res}
Using a similar strategy to Lemma \ref{lem:nlf_lem5} and Corollary \ref{cor:nlf_cor1}, one can establish the following under (D\ref{hyp_diff:1}). For any $(n,r)\in\mathbb{N}_0\times\mathbb{N}$, there exists a $C<+\infty$ such that for any $(l,\varphi,x)\in\mathbb{N}\times\mathcal{B}_b(\mathbb{R}^{d_x})\cap\textrm{\emph{Lip}}_{\|\cdot\|_2}(\mathbb{R}^{d_x})\times\mathbb{R}^{d_x}$:
$$
\overline{\mathbb{E}}\Big[\Big|\overline{\mathbb{E}}[\varphi(\widetilde{X}_{n+1}^x)Z_{n}^l(\widetilde{X}_0^x,\widetilde{X}_{\Delta_l}^x,\dots,\widetilde{X}_{n-\Delta_l}^x)|\mathcal{Y}_n] - 
\overline{\mathbb{E}}[\varphi(\widetilde{X}_{n+1}^x)Z_{n}^{l-1}(\widetilde{X}_0^x,\widetilde{X}_{\Delta_{l-1}}^x,\dots,\widetilde{X}_{n-\Delta_{l-1}}^x)|\mathcal{Y}_n]\Big|^r\Big]^{1/r} 
$$
\begin{equation}
\leq C(\|\varphi\|+\|\varphi\|_{\textrm{\emph{Lip}}})\Delta_l^{1/2}.
\end{equation}
For any $(n,r)\in\mathbb{N}^2$, there exists a $C<+\infty$ such that for any $(l,\varphi,x)\in\mathbb{N}\times\mathcal{B}_b(\mathbb{R}^{d_x})\cap\textrm{\emph{Lip}}_{\|\cdot\|_2}(\mathbb{R}^{d_x})\times\mathbb{R}^{d_x}$:
$$
\overline{\mathbb{E}}\Bigg[\Bigg|\frac{\overline{\mathbb{E}}[\varphi(\widetilde{X}_{n+1}^x)Z_{n+1}^l(\widetilde{X}_0^x,\widetilde{X}_{\Delta_l}^x,\dots,\widetilde{X}_{n+1-\Delta_l}^x)|\mathcal{Y}_{n+1}]}
{\overline{\mathbb{E}}[Z_{n}^l(\widetilde{X}_0^x,\widetilde{X}_{\Delta_l}^x,\dots,\widetilde{X}_{n-\Delta_l}^x)|\mathcal{Y}_{n}]}-
\frac{\overline{\mathbb{E}}[\varphi(\widetilde{X}_{n+1}^x)Z_{n+1}^{l-1}(\widetilde{X}_0^x,\widetilde{X}_{\Delta_{l-1}}^x,\dots,\widetilde{X}_{n+1-\Delta_{l-1}}^x)|\mathcal{Y}_{n+1}]}
{\overline{\mathbb{E}}[Z_{n}^{l-1}(\widetilde{X}_0^x,\widetilde{X}_{\Delta_{l-1}}^x,\dots,\widetilde{X}_{n-\Delta_{l-1}}^x)|\mathcal{Y}_{n}]}
\Bigg|^r\Bigg]^{1/r}
$$
$$
\leq C(\|\varphi\|+\|\varphi\|_{\textrm{\emph{Lip}}})\Delta_l^{1/2}.
$$
For any $(n,r)\in\mathbb{N}_0\times\mathbb{N}$, there exists a $C<+\infty$ such that for any $(l,\varphi,x)\in\mathbb{N}\times\mathcal{B}_b(\mathbb{R}^{d_x})\cap\textrm{\emph{Lip}}_{\|\cdot\|_2}(\mathbb{R}^{d_x})\times\mathbb{R}^{d_x}$:
$$
\overline{\mathbb{E}}\Big[\Big|\overline{\mathbb{E}}[\varphi(\widetilde{X}_{n+1}^x)Z_{n+1}^l(\widetilde{X}_0^x,\widetilde{X}_{\Delta_l}^x,\dots,\widetilde{X}_{n+1-\Delta_l}^x)|\mathcal{Y}_{n+1}] - 
\overline{\mathbb{E}}[\varphi(\widetilde{X}_{n+1}^x)Z_{n+1}^{l-1}(\widetilde{X}_0^x,\widetilde{X}_{\Delta_{l-1}}^x,\dots,\widetilde{X}_{n+1-\Delta_{l-1}}^x)|\mathcal{Y}_{n+1}]\Big|^r\Big]^{1/r} 
$$
$$
\leq C(\|\varphi\|+\|\varphi\|_{\textrm{\emph{Lip}}})\Delta_l^{1/2}.
$$
\end{rem}

\subsection{Results for the Particle Filter}\label{app:pf}

\begin{lem}\label{lem:res1}
Assume (D\ref{hyp_diff:1}). Then for any $(p,n,r)\in\mathbb{N}_0^2\times\mathbb{N}$, $n\geq p$, there exists a $C<+\infty$ such that for any $l\in\mathbb{N}_0$, $(i,\varphi)\in\{1,\dots,N\}\times\mathcal{B}_b(E_l)$
$$
\max\{\overline{\mathbb{E}}[\pi_p^l(\mathbf{Q}_{p,n}^l(\varphi))^{-r}],\overline{\mathbb{E}}[\pi_p^{l,N}(\mathbf{Q}_{p,n}^l(\varphi))^{-r}],\overline{\mathbb{E}}[\mathbf{Q}_{p,n}^l(\varphi)(U_p^{l,i})^r]^{1/r}\} \leq C\|\varphi\|.
$$
\end{lem}

\begin{proof}
We start by considering $\overline{\mathbb{E}}[\mathbf{Q}_{p,n}^l(\varphi)(U_p^{l,i})^r]^{1/r}$. Applying Jensen's inequality, we have the upper-bound:
$$
\overline{\mathbb{E}}[\mathbf{Q}_{p,n}^l(\varphi)(U_p^{l,i})^r] \leq \|\varphi\|^r\overline{\mathbb{E}}[\mathbf{G}_p^l(U_p^{l,i})^r \prod_{q=p+1}^{n-1} \mathbf{G}_q^l(U_q)^r]
$$
where $U_{p+1},\dots,U_{n}$ is a Markov chain of initial distribution $M^l(U_p^{l,i},\cdot)$ and transition $M^l$. As $h$ is bounded, we have the upper-bound
\begin{eqnarray*}
\overline{\mathbb{E}}[\mathbf{Q}_{p,n}^l(\varphi)(U_p^{l,i})^r] & \leq & C\|\varphi\|^r\overline{\mathbb{E}}\Big[
\exp\Big\{r\sum_{s_1=0}^{\Delta_l^{-1}-1}\sum_{s_2=1}^{d_y}h^{(s_2)}(X_{p+s_1\Delta_l}^{l,i})(Y_{p+(s_1+1)\Delta_l}^{(s_2)}-Y_{p+s_1\Delta_l}^{(s_2)})\Big\}\times
\\ & & 
\exp\Big\{r\sum_{q=p+1}^{n-1}\sum_{s_1=0}^{\Delta_l^{-1}-1}\sum_{s_2=1}^{d_y}h^{(s_2)}(X_{q+s_1\Delta_l}^l)(Y_{q+(s_1+1)\Delta_l}^{(s_2)}-Y_{q+s_1\Delta_l}^{(s_2)})\Big\}
\Big].
\end{eqnarray*}
Taking expectations w.r.t.~the process $\{Y_t\}$ we have
\begin{eqnarray*}
\overline{\mathbb{E}}[\mathbf{Q}_{p,n}^l(\varphi)(U_p^{l,i})^r] & \leq &  C\|\varphi\|^r\overline{\mathbb{E}}\Big[
\exp\Big\{\frac{r^2}{2}\sum_{s_1=0}^{\Delta_l^{-1}-1}\sum_{s_2=1}^{d_y}h^{(s_2)}(X_{p+s_1\Delta_l}^{l,i})^2\Delta_l\Big\}\times
\\ & & 
\exp\Big\{\frac{r^2}{2}\sum_{q=p+1}^{n-1}\sum_{s_1=0}^{\Delta_l^{-1}-1}\sum_{s_2=1}^{d_y}h^{(s_2)}(X_{q+s_1\Delta_l}^l)^2\Delta_l\Big\}
\Big].
\end{eqnarray*}
Then using the fact that $h$ is bounded, it clearly follows
$$
\overline{\mathbb{E}}[\mathbf{Q}_{p,n}^l(\varphi)(U_p^{l,i})^r]^{1/r} \leq C\|\varphi\|.
$$
For the terms $\overline{\mathbb{E}}[\pi_p^l(\mathbf{Q}_{p,n}^l(\varphi))^{-r}]$ and $\overline{\mathbb{E}}[\pi_p^{l,N}(\mathbf{Q}_{p,n}^l(\varphi))^{-r}]$ one can apply (the conditional) Jensen's inequality and essentially the same argument as above and hence the proof is omitted.
\end{proof}

\begin{prop}\label{prop:lp_pf}
Assume (D\ref{hyp_diff:1}). Then for any $(p,n,r)\in\mathbb{N}_0^2\times\mathbb{N}$, $n\geq p$, there exists a $C<+\infty$ such that for any $l\in\{0,1,\dots\}$, $(N,\varphi)\in\mathbb{N}\times\mathcal{B}_b(E_l)$
$$
\overline{\mathbb{E}}[|(\pi_p^{l,N}-\pi_p^l)(\mathbf{Q}_{p,n}^l(\varphi))|^r]^{1/r} \leq \frac{C\|\varphi\|}{\sqrt{N}}.
$$
\end{prop}

\begin{proof}
The proof is by induction on $p$ for any fixed $n\geq p$. If $p=0$ one can apply the conditional Marcinkiewicz-Zygmund inequality to yield:
$$
\overline{\mathbb{E}}[|(\pi_p^{l,N}-\pi_p^l)(\mathbf{Q}_{p,n}^l(\varphi))|^r]^{1/r} \leq \frac{C}{\sqrt{N}} \overline{\mathbb{E}}[\mathbf{Q}_{p,n}^l(\varphi)(U_p^{l,1})^r]^{1/r}.
$$
Then one has the result by Lemma \ref{lem:res1}.

For the induction step, we have the standard decomposition via Minkowski
\begin{equation}\label{eq:master}
\overline{\mathbb{E}}[|(\pi_p^{l,N}-\pi_p^l)(\mathbf{Q}_{p,n}^l(\varphi))|^r]^{1/r} \leq  T_1 + T_2 + T_3
\end{equation}
where
\begin{eqnarray*}
T_1 & = & \overline{\mathbb{E}}[|(\pi_p^{l,N}-\Phi_p^l(\pi_{p-1}^{l,N}))(\mathbf{Q}_{p,n}^l(\varphi))|^r]^{1/r}\\
T_2 & = & \overline{\mathbb{E}}\Big[\Big|\frac{\Phi_p^l(\pi_{p-1}^{l,N})(\mathbf{Q}_{p,n}^l(\varphi))}{\pi_{p-1}^l(\mathbf{G}_{p-1}^l)}\{(\pi_{p-1}^l-\pi_{p-1}^{l,N})(\mathbf{G}_{p-1}^l)\}\Big|^r\Big]^{1/r} \\
T_3 & = & \overline{\mathbb{E}}\Big[\Big|\frac{(\pi_{p-1}^{l,N}-\pi_{p-1}^{l})(\mathbf{Q}_{p-1,n}^l(\varphi))}{\pi_{p-1}^l(\mathbf{G}_{p-1}^l)}\Big|^r\Big]^{1/r}.
\end{eqnarray*}
By the same argument as for the initialization
\begin{equation}\label{eq:t1}
T_1 \leq \frac{C\|\varphi\|}{\sqrt{N}}.
\end{equation}
For $T_2$ applying H\"older
$$
T_2 \leq \overline{\mathbb{E}}[\pi_{p-1}^l(\mathbf{G}_{p-1}^l)^{-3r}]^{1/(3r)}\overline{\mathbb{E}}[\Phi_p^l(\pi_{p-1}^{l,N})(\mathbf{Q}_{p,n}^l(\varphi))^{3r}]^{1/(3r)}\overline{\mathbb{E}}[|(\pi_{p-1}^l-\pi_{p-1}^{l,N})(\mathbf{G}_{p-1}^l)|^{3r}]^{1/(3r)}.
$$
For the left-most term on the R.H.S.~one can apply Lemma \ref{lem:res1}. For the middle term on the R.H.S.~one can apply the conditional Jensen inequality and Lemma \ref{lem:res1}. For the right-most term on the R.H.S.~one can apply the induction hypothesis. Hence
\begin{equation}\label{eq:t2}
T_2 \leq \frac{C\|\varphi\|}{\sqrt{N}}.
\end{equation}
For $T_3$, one can use Cauchy-Schwarz, Lemma \ref{lem:res1} and the induction hypothesis to yield
\begin{equation}\label{eq:t3}
T_3 \leq \frac{C\|\varphi\|}{\sqrt{N}}.
\end{equation}
Combining \eqref{eq:t1}-\eqref{eq:t3} with \eqref{eq:master} concludes the proof.
\end{proof}

\begin{rem}\label{rem:idiot_referee}
It is simple to extend Proposition \ref{prop:lp_pf} to the case of the filter. That is, for any $(p,n,r)\in\mathbb{N}_0^2\times\mathbb{N}$, $n\geq p$, there exists a $C<+\infty$ such that for any $l\in\{0,1,\dots\}$, $(N,\varphi)\in\mathbb{N}\times\mathcal{B}_b(E_l)$
$$
\overline{\mathbb{E}}[|(\eta_p^{l,N}-\eta_p^l)(\mathbf{Q}_{p,n}^l(\varphi))|^r]^{1/r} \leq \frac{C\|\varphi\|}{\sqrt{N}}.
$$
\end{rem}

\begin{rem}\label{rem:ext_lp_pf}
It is straightforward to extend Proposition \ref{prop:lp_pf} to the following result, under (D\ref{hyp_diff:1}): for any $(p,n,r)\in\mathbb{N}_0^2\times\mathbb{N}$, $n\geq p$, there exists a $C<+\infty$ such that for any $l\in\mathbb{N}_0$, $(N,\varphi)\in\mathbb{N}\times\mathcal{B}_b(E_l)$
$$
\overline{\mathbb{E}}[|(\pi_p^{l,N}-\pi_p^l)(\mathbf{Q}_{p,n}^l(\mathbf{G}_{n}^l\varphi))|^r]^{1/r} \leq \frac{C\|\varphi\|}{\sqrt{N}}.
$$
\end{rem}

\begin{cor}\label{cor:pf_cor}
Assume (D\ref{hyp_diff:1}). Then for any $(p,n,r)\in\mathbb{N}_0^2\times\mathbb{N}$, $n\geq p$, there exists a $C<+\infty$ such that for any $l\in\mathbb{N}_0$, $(N,\varphi)\in\mathbb{N}\times\mathcal{B}_b(E_l)$
$$
\overline{\mathbb{E}}[|\pi_p^{l,N}(\mathbf{D}_{p,n}^l(\mathbf{G}_{n}^l\varphi))|^r]^{1/r} \leq \frac{C\|\varphi\|}{\sqrt{N}}.
$$
\end{cor}

\begin{proof}
Noting that $\pi_p^l(\mathbf{D}_{p,n}^l(\mathbf{G}_{n}^l\varphi))=0$ a.s., the result follows immediately by Cauchy-Schwarz, Lemma \ref{lem:res1} and Proposition \ref{prop:lp_pf} (see Remark \ref{rem:ext_lp_pf}).
\end{proof}

\end{document}